\documentclass[12pt,a4paper]{amsart}

\usepackage{amsmath}
\usepackage{enumerate}
\usepackage{graphicx}
\usepackage{amsfonts}
\usepackage{amsthm}
\usepackage{amssymb}
\usepackage{esint}
\usepackage{bbm}
\usepackage[latin1]{inputenc}
\usepackage{hyperref}
\usepackage{microtype}
\usepackage{ifpdf}

\theoremstyle{plain}
\newtheorem{lemma}{Lemma}[section]
\newtheorem{prop}[lemma]{Proposition}
\newtheorem{coro}[lemma]{Corollary}
\newtheorem{theorem}[lemma]{Theorem}

\theoremstyle{definition}
\newtheorem{definition}[lemma]{Definition}
\newtheorem{remark}[lemma]{Remark}

\newcommand{\ts}{\hspace{0.5pt}}

\newcommand{\RR}{\mathbb{R}\ts}

\newcommand{\NN}{\mathbb{N}}

\newcommand{\aV}[1]{\left\Vert #1\right\Vert}
\newcommand{\as}[1]{\langle #1\rangle}

\newcommand{\ow}[1]{\widetilde{#1}}

\newcommand{\sumv}[1]{\sum_{#1\in V}}
\newcommand{\QN}{Q^{(N)}} 
\newcommand{\QD}{Q^{(D)}} 
\newcommand{\LN}{L^{(N)}} 
\newcommand{\LD}{L^{(D)}} 

\newcommand{\Hmm}[1]{\leavevmode{\marginpar{\tiny%
$\hbox to 0mm{\hspace*{-0.5mm}$\leftarrow$\hss}%
\vcenter{\vrule depth 0.1mm height 0.1mm width \the\marginparwidth}%
\hbox to 0mm{\hss$\rightarrow$\hspace*{-0.5mm}}$\\\relax\raggedright #1}}}

\begin{document}

\title{Global properties of Dirichlet forms on discrete spaces}

\author[]{Marcel Schmidt}
\email{schmidt.marcel@uni-jena.de}
\address{Mathematisches Institut, Friedrich Schiller Universit\"at Jena,
  07743 Jena, Germany \newline http://www.analysis-lenz.uni-jena.de/Team/Marcel+Schmidt.html}

\begin{abstract}

We provide an introduction to Dirichlet forms on discrete spaces and study their global properties such as recurrence,
stochastic completeness and regularity of the Neumann form. In this setting we compare the notion of 
a recurrent Dirichlet form and a recurrent discrete time Markov chain of a given graph. We prove several known and several new characterizations of recurrence  by using functional analytic Dirichlet form methods only. Finally, we compare all the mentioned global properties and discuss their relation to spectral theory.

\end{abstract}

\date{\today} %
\maketitle

\setlength{\parskip}{0pt}

\tableofcontents

\setlength{\parskip}{5pt}

\section*{Introduction}

 The study of  the long time behavior of sample paths of discrete time random walks on countable sets is a very well established area of mathematical research. Probably the most fundamental question in this direction is whether a given random walk is returning to each point infinitely often or not, i.e., whether the random walk is recurrent or transient. When dealing with the special case of Markov chains it is now well-known that these questions are intimately linked to the study of certain linear operators on function spaces over the state space, see e.g. the textbooks \cite{DS,FE, Soa, Woe}. If, additionally, the Markov chains are reversible the corresponding operators satisfy a further symmetry condition making them accessible to the theory of self-adjoint operators. These operators are sometimes referred to as discrete Laplacians. Independently, they have gained a lot of attention outside of probability theory in the realm of spectral geometry and spectral theory. 

Starting with the work of Yamasaki \cite{Yam1,Yam2} it was realized that potential theory also provides useful tools for studying the long-time behavior of discrete-time Markov chains. In the simplest case the potential theory is induced by a certain quadratic form coming from a weighted graph on the state space. The weights of the graph and the random walk correspond through its transition probabilities. It can be seen, that this approach  is indeed a special case of the potential theory of Dirichlet spaces as presented in \cite{FOT}. However, there the authors are interested in the study of symmetric Markov processes in continuous time. On countable state spaces this means that additional holding times at each site are introduced. As a consequence, the local behavior, i.e., the holding times and the transition probabilities, does not necessarily determine the process uniquely and the process might have a finite lifetime (be stochastically incomplete). 

One of the advantages of the approach of \cite{FOT} is that the so-called global properties discussed above can be formulated in purely functional analytic terms.  In \cite{KL}  Keller and Lenz quite recently introduced the full power of the theory of general regular Dirichlet forms of \cite{FOT} to the study of discrete Laplacians. They were inspired by the need of a convenient framework for the investigation of unbounded Laplacians on discrete spaces which recently gained considerable attention starting with the works \cite{Kel1, Web, Woj}. Keller and Lenz study the relation of spectral theory with the uniqueness of associated processes and their lifetime in quite some detail. However, a treatment of recurrence and transience is missing in their exposition. 

This is where the present text comes into play. It is intended to provide a convenient and self-contained study of recurrence and transience within the theory of Dirichlet forms on discrete spaces.  In particular, it provides an introduction to the theory of discrete Dirichlet spaces. 

As discussed above, recurrence and transience of random walks are
usually investigated  in discrete time and, when passing to
continuous time, other global properties enter the game. Thus, the
study of these properties via Dirichlet forms raises three main
questions: 

\begin{itemize}
 \item[(a)]  How is the notion of recurrence/transience for discrete time random walks related to the one for regular Dirichlet forms?
 \item[(b)] How can one use the theory of regular Dirichlet forms to prove known and new criteria for recurrence/transience? 
 \item[(c)] How is recurrence related to the uniqueness of the processes and their lifetime? 
\end{itemize}

The present article provides answers to all three questions. 

Question (a) is answered by Theorem \ref{discrete time v.s. continuous time}. It establishes an explicit formula which relates the Green's function of a discrete time random walk on a given graph and the Green's function of the regular Dirichlet form associated with this graph. In particular, it shows that both notions of recurrence agree. While the last result is certainly known to experts, a reference proving it in the full generality of the framework presented here seems to be missing. 

Given the previously discussed answer to  Question (a), Question (b) makes sense. For an answer the known criteria of recurrence and transience which deal with the existence of monopoles of finite finite energy (Theorem \ref{monopol}), the existence of superharmonic functions of finite energy (Theorem \ref{superharmonic}) and the capacity of points (Theorem \ref{potential}) are proven within the framework provided by \cite{FOT}. Two new criteria which are inspired by the works \cite{GM, JP} and deal with the vanishing of a boundary term in the discrete version of Green's formula are obtained (Theorem~\ref{integral recurrence} and Theorem~\ref{boundary recurrence}). Finally, the connection of recurrence to the spectral theory of the discrete Laplacian on the space of functions of finite energy is discussed in Theorem \ref{last}.  

Question (c) is answered at the end of the paper. Uniqueness of processes is addressed in analytic terms by asking when the Neumann form  associated with a graph is regular while the lifetime problem is studied analytically in terms of stochastic completeness of the regular Dirichlet form, see Section~\ref{further global properties}.  It turns out that regularity of the Neumann form and stochastic completeness are also related to the validity of Green's formula (Theorem  \ref{integral sc} and Theorem \ref{unique solvability QN=QD}) on certain function spaces. Furthermore, it is shown that recurrence always implies the regularity of the Neumann form and stochastic completeness (Theorem \ref{implications of recurrence}) and that all these global properties are equivalent when the underlying measure is finite (Theorem \ref{finite measure}). 

The paper is organized as follows. Section \ref{forms and spaces associated with graphs} introduces the basic objects of investigation, namely weighted graphs and an ensemble of associated forms, operators and spaces. In this presentation we basically follow \cite{HKLW,KL}. Section \ref{general theory} outlines the theory of recurrence and transience of Dirichlet forms as it is presented in \cite[Section~1.5]{FOT}. The contents of this section are certainly well-known. However, as the known proofs simplify substantially in the discrete setting we include them for the convenience of the reader. Section \ref{discrete time versus continuous time} compares the notion of recurrence of discrete time Markov chains and Dirichlet forms associated with graphs. In Section \ref{recurrence and transience} we discuss the consequences of the theory developed in Section \ref{general theory} when applied to a regular Dirichlet form on a discrete space. We recover known characterizations of recurrence purely by Dirichlet form methods and even obtain new ones by the use of these techniques. Section \ref{further global properties} introduces two further global properties, namely stochastic completeness and regularity of the Neumann form. We prove several characterizations of these properties. Section \ref{consequences of recurrence} deals with the connection of recurrence and the other global properties and with the connection of recurrence and spectral theory of the discrete Laplacian.

{\bf Remark on the history of this paper:} A preliminary version of this paper \cite{Sch} was published on the arXiv. It was not sent to a refereed journal since it is rather long and, besides new theory, contains some known results. However, this exposition turned out to close the gap between the textbooks \cite{FOT} and \cite{Soa, Woe} and seems to be quite useful. In fact, the way the theory is presented here is used in the publications \cite{GHKLW,HKLMS,HuK,Kel,KLSW,KLSWoi} which cite \cite{Sch}. 

{\bf Acknowledgements:} I am grateful to  Daniel Lenz for introducing me to this area and for the help when preparing this manuscript. I thank Matthias Keller and Sebastian Haeseler for the many fruitful discussion on the subject and Wolfgang Woess for giving useful comments on an earlier version of the text. Furthermore, I acknowledge the financial support of the Graduiertenkolleg 1523/2 : Quantum and gravitational fields.


\section{Forms and spaces associated with graphs} \label{forms and spaces associated with graphs}

In this chapter we introduce the objects of our studies. Following \cite{HKLW} we specify what we will call a weighted graph $(b,c)$ over a vertex set $V$, define the ensemble $(\ow{D},\ow{F},\ow{L},\ow{Q})$ of associated objects and show their basic connections. In the next section we equip $\ow{D}$ with an inner product which turns it into a Hilbert space. Since spaces of this sort were introduced in \cite{Yam1}, we call it the Yamasaki space $\mathbf{D}$. We prove some results about its structure. Finally, we introduce the notion of a Dirichlet form associated with $(b,c)$. Most results of this chapter are well known and we include their proofs for the convenience of the reader.

\subsection{Basic definitions}

Assume that $V$ is an infinite, countable set. Let $b:V\times V \to [0,\infty)$ such that 
\begin{itemize}
  \item [(b0)] $b(x,x) = 0$ for all $x \in V$,
  \item [(b1)] $b(x,y)=b(y,x)$ for all $x,y\in V$,
  \item [(b2)] $\sum_{y\in V} b(x,y)<\infty$ for all $x\in V$,
\end{itemize}
and let $c: V \to [0,\infty).$ We call the pair $(b,c)$ a weighted graph over the vertex set $V$. We say $x,y \in V$ are connected by an edge whenever $b(x,y)> 0$. In this case we write $x\sim y$ and call $b(x,y)$ the weight of the edge connecting $x$ and $y$. Vertices $x \in V$ with $c(x) > 0$ might be thought of being connected with a point $\infty$ which is not contained in $V$. We call a finite sequence of vertices $x_0,\ldots,x_n \in V$ a \emph{path} connecting $x_0$ and $x_n$ if $x_j \sim x_{j+1}$ for $j=0,\ldots,n-1$. A subset $W \subseteq V$ is said to be \emph{connected} if for every $x,y \in W$ there is a path in $W$ connecting $x$ and $y$. The quantity
$$\text{deg}(x) := \sum_{y\in V}b(x,y) + c(x)$$
is said to be the \emph{generalized vertex degree} of $x$. For a graph $(b,0)$ with $b \in \{0,1\}$ the number $\text{deg}(x)$ coincides with the number of edges emerging from $x$. A graph $(b,c)$ over $V$ is called \emph{locally finite} if the sets 
$$\{y \in V| \, \, b(x,y) > 0\}$$
are finite for every $x \in V$, i.e., each vertex is only connected with finitely many other vertices.

Let $C(V)$ be the set of all real-valued functions on $V$ and let $C_c(V)$ be the subset of all real-valued functions of finite support. To a graph $(b,c)$ over $V$ we associate the quadratic form 
$$\ow{Q}:= \widetilde{Q}_{b,c}: C(V) \to [0,\infty],$$
which acts by
$$ \ow{Q}(u) := \frac{1}{2}\sum_{x,y \in V} b(x,y)(u(x)-u(y))^2  + \sum_{x\in V} u(x)^2 c(x).$$  

The next lemma is crucial for further discussions. It shows that $\ow{Q}$ satisfies certain cut-off properties.

\begin{lemma} \label{Markov property}
Let $F: \RR \to \RR$ be a normal contraction (i.e.  $F(0) = 0$ and $|F(x)-F(y)| \leq |x-y|$ for all $x,y \in \RR$). For $u \in C(V)$ the inequality
$$\ow{Q}(F \circ u) \leq \ow{Q}(u)$$
holds. 
\end{lemma} 
\begin{proof} A direct calculation yields the statement.
\end{proof}

\begin{remark}
\begin{itemize}
\item Typical examples for normal contractions are 
$F(x) = |x|$ and $F(x) = (0\vee x)\wedge 1$ (here  $a \vee b = \max\{a,b\}$ and $a \wedge b = \min\{a,b\}$).
\item The above lemma is important for showing that certain restrictions of $\ow{Q}$ are Dirichlet forms.
\end{itemize}
\end{remark}
We will be interested in the space of all \emph{functions of finite energy}, which is defined as 
$$\ow{D} := \{u \in C(V)|\, \, \widetilde{Q}(u) < \infty \}. $$
For $x \in V$ let $\delta_x$ be the function on $V$ which vanishes everywhere except in $x$, where it takes value $1$. Obviously the equality
$$\ow{Q}(\delta_x) = \text{deg}(x)$$
holds. Therefore, assumption (b2) implies that $C_c(V)$, the space of finitely supported functions, is contained in $\widetilde{D}$. 

Abusing notation we extend $\widetilde{Q}$ to a bilinear map 
$$\ow{Q}:\widetilde{D}\times \widetilde{D} \to \RR$$
acting by
$$\widetilde{Q}(u,v) := \frac{1}{2}\sum_{x,y \in V} b(x,y)(u(x)-u(y)) (v(x)-v(y)) + \sum_{x\in V} u(x)v(x) c(x).$$
Note that the above sum is absolutely convergent by the definition of $\widetilde{D}$. 
\begin{remark}
As for $\ow{Q}$ we will usually write $B(u)$ instead of $B(u,u)$ when dealing with a bilinear from $B$. 
\end{remark}
We come  to the \emph{formal operator} $\ow{L}$ associated with $\widetilde{Q}$. Let 
$$\ow{F} = \{u:V \to \RR \, | \sum_{y \in V}b(x,y)|u(y)| < \infty \text{ for all } x\in V \}$$
and $m:V \to (0,\infty)$. We   define 
$$\ow{L}:= \ow{L}_{b,c,m}: \ow{F} \to C(V)$$
via
$$\ow{L}u (x) := \frac{1}{m(x)}\sum_{y\in V}b(x,y)(u(x)-u(y)) + \frac{c(x)}{m(x)}u(x).$$
The definition of $\ow{F}$ and (b2) ensure that the sum is absolutely convergent. The following lemma is the crucial link between $\ow{L}$ and $\ow{Q}$. For various versions of this statement, see e.g. \cite{HK, HKLW, KL}. 
\begin{lemma}[Green's formula]  \label{green} For all $u \in \ow{D}$ and $v \in C_c(V)$ the equality
$$\ow{Q}(u,v) = \sum_{x\in V} (\ow{L}u)(x)v(x)m(x) = \sum_{x\in V} u(x)(\ow{L}v)(x)m(x)$$
holds.

\end{lemma}
\begin{proof}
The proof will be done in two steps. First we show $\ow{D} \subseteq \ow{F}$ following Proposition 2.8 of \cite{HKLW}. As a second step we prove the desired equality as in Lemma 4.7 of \cite{HK}.

\emph{Step 1}: Let $u\in \ow{D}$. Using Cauchy-Schwarz inequality we obtain
\begin{align*}
	\sum_{y\in V}& b(x,y)|u(y)| \leq \sum_{y\in V}b(x,y)|u(x)-u(y)| + \sum_{y\in V}b(x,y)|u(x)|\\
	&\leq \left(\sum_{y\in V}b(x,y)\right)^{1/2} \left(\sum_{y\in V}b(x,y)|u(x)-u(y)|^2 \right)^{1/2} + \text{deg}(x)|u(x)| \\
	&\leq \text{deg}(x)^{1/2} \ow{Q}(u)^{1/2} + \text{deg}(x)|u(x)| < \infty.
\end{align*}
This shows $u \in \ow{F}$.

\emph{Step 2}: Let $u \in \ow{D}$ and $v \in C_c(V)$. Step 1 and (b1) yield
$$\sum_{x,y \in V} b(x,y)|u(x)v(y)| = \sum_{y\in V} |v(y)|\sum_{x\in V} b(x,y)|u(x)| < \infty.$$
Moreover, by condition (b2) we obtain
$$\sum_{x,y \in V} b(x,y)|u(x)v(x)| = \sum_{x\in V} |v(x)||u(x)|\sum_{y\in V} b(x,y) < \infty.$$
This allows us to rearrange the summation of 
$$  \sum_{x\in V}\left( \sum_{y\in V} b(x,y)(u(x)-u(y))v(x) + c(x)u(x)v(x)\right)$$
and the statement follows by a simple computation.
\end{proof}

The next lemma is standard. It shows some continuity of differences of function values with respect to $\ow{Q}$, see e.g. \cite[Lemma~2.5]{JP}.

\begin{lemma} \label{difference functional}
Let $(b,c)$ be connected and  $x,y \in V$. Then, there exists a constant $K_{x,y} > 0$  such that for every $u \in \ow{D}$  the inequality
$$|u(x)-u(y)| \leq K_{x,y} \, \ow{Q}(u)^{1/2} $$
holds. 
\end{lemma}  
\begin{proof}
Let $x = x_0,...,x_n=y$ be a path of pairwise different vertices connecting $x$ and $y$. Set 
$$K_{x,y} =\left(\sum_{j=1}^n \frac{1}{b(x_{j-1},x_j)} \right)^{1/2}.$$
By Cauchy-Schwarz inequality we infer
\begin{align*}
	|u(x)-u(y)| &\leq \sum_{j=1}^n |u(x_j)-u(x_{j-1})| \\ 
	&\leq K_{x,y} \, \left(\sum_{j=1}^n b(x_{j-1},x_j) |u(x_j)-u(x_{j-1})|^2 \right)^{1/2}.\\
\end{align*}
Since the $x_j$ were assumed to be pairwise different the sum in the last term of the above inequality can be estimated by $\ow{Q}$. This yields the statement. 
\end{proof}


\subsection{The Yamasaki space}

Following \cite{Soa}, we introduce a Hilbert space associated with a connected graph $(b,c)$. Fix some reference point $o \in V$. For $u,v \in \ow{D}$ we define the inner product
$$ \as{u,v}_o := \ow{Q}(u,v) + u(o)v(o).$$
Let $\|\cdot\|_o$ be the corresponding norm which is non-degenerate due to the connectedness of $(b,c)$. The following proposition is part of Lemma 3.14 and Theorem 3.15 in \cite{Soa}.
\begin{prop}[Properties of $(\ow{D},\as{\cdot,\cdot}_o)$] \label{structure of D} 
Let $(b,c)$ be a connected graph over $V$. The following statements hold.
\begin{itemize}
 \item[(a)] For every $x \in V$ the functional $F_x: \ow{D} \to \RR$, $F_x[u] =  u(x)$ is continuous with respect to $\aV{\cdot}_o$. In particular, for $o, o' \in V$ the norms $\aV{\cdot}_o$ and $\aV{\cdot}_{o'}$ are equivalent.
 \item[(b)] $\ow{D}$ equipped with $\as{\cdot,\cdot}_o$ is a Hilbert space.  
\end{itemize}
\end{prop}
\begin{proof}
(a) We only need to prove the first statement. The 'In particular' part follows from the first statement as it shows that $u \mapsto u(o)$ is continuous with respect to $\aV{\cdot}_{o'}$. 

Let $x \in V$ be given. Choose a path $o = x_0,...,x_n = x$ of pairwise different points from $o$ to $x$ and let $K_{x_{i-1},x_i}$ be constants for these vertices as in Lemma \ref{difference functional}. We obtain
\begin{align*}
	|u(x)| &\leq \sum_{i=1}^n |u(x_{i-1})-u(x_i)| + |u(o)|\\
	& \leq \ow{Q}(u)^{1/2} \sum_{i=1}^n K_{x_{i-1},x_i} + \aV{u}_o.
\end{align*} 
This finishes the proof of (a).

(b) It suffices to show completeness. Let $(u_n)$ be a Cauchy sequence in $(\ow{D},\as{\cdot,\cdot}_o)$. By (a) the sequence $(u_n)$ converges pointwise towards a function $u$. Fatou's lemma yields
$$ \ow{Q}(u-u_n) \leq \liminf_{l\to \infty} \ow{Q}(u_l-u_n),$$
which can be made small by choosing $n$ large enough. This shows $u \in \ow{D}$ and $u_n \to u$ with respect to $\aV{\cdot}_o$.
\end{proof}

\begin{definition}[Yamasaki space]
 The pair $(\ow{D},\as{\cdot ,\cdot}_o)$ is called the \emph{Yamasaki space} associated with $(b,c)$. We write $\mathbf{D}$ for short whenever we refer to $\ow{D}$ endowed with the topology generated by $\aV{\cdot}_o$.  The set $\mathbf{D}_0$ is the closed subspace of $\mathbf{D}$ given by
$$\mathbf{D}_0 := \overline{C_c(V)}^{\aV{\cdot}_o}.$$ 
\end{definition}

\begin{remark}
 Lemma \ref{structure of D} shows that the topology on $\mathbf{D}$ and, in particular, the space $\mathbf{D}_0$ do not depend on the choice of the reference point $o \in V$.
\end{remark}
For our further discussion, we will need some more approximation results based on the following theorem. The presented proof was suggested by Daniel Lenz.
\begin{theorem}[Convergence in $\mathbf{D}$] \label{convergence in D}
Let $u_n, u \in \mathbf{D}$ be given. The following assertions are equivalent.
\begin{itemize}
 \item[(i)] $\|u_n - u\|_o \to 0$, as $n \to \infty$.
 \item[(ii)]$\limsup_{n\to \infty} \ow{Q}(u_n) \leq \ow{Q}(u) \text{ and } u_n \to u \text{ pointwise, as } n\to \infty.$
\end{itemize}
\end{theorem}

\begin{proof}
The implication  (i) $\Rightarrow$ (ii) follows from Proposition \ref{structure of D} (a). 

For poving implication (ii) $\Rightarrow$ (i) we let $u_n, u$ be given such that $u_n \to u$ pointwise and $\limsup\ow{Q}(u_n) \leq \ow{Q}(u)$. These two conditions imply that $(u_n)$ is a bounded sequence in $\mathbf{D}$. Since $\mathbf{D}$ is a Hilbert space every ball is weakly compact, thus every subsequence of $(u_n)$ has a weakly convergent subsequence. As $(u_n)$ converges pointwise towards $u$, all the occurring limits must coincide. We infer $u_n \to u$ weakly in $\mathbf{D}$. With this observation the desired statement follows from the inequality
\begin{align*}
0&\leq \ow{Q}(u-u_n) +(u(o)-u_n(o))^2\\
 &= \ow{Q}(u) + u(o)^2 + \ow{Q}(u_n) + u_n(o)^2 - 2\as{u,u_n}_o,
\end{align*} 
after taking $\limsup$.
\end{proof}

\begin{coro}\label{bounded convergence}
Let $(b,c)$ be connected and  $u\in \mathbf{D}$ be given. For any natural number $N$ the function $u_N := ((-N)\vee u)\wedge N$ belongs to $\mathbf{D}$ and $\|u_N - u\|_o \to 0$, as $N \to \infty$.
\end{coro}

\begin{proof}
Lemma \ref{Markov property} shows $\ow{Q} (u_N) \leq \ow{Q}(u)$. We infer the statement from the previous theorem.
\end{proof}

\begin{coro} \label{small convergence}
Let $(b,c)$ be connected and assume $c\equiv 0$. Let $(e_n)$ be a sequence in $\mathbf{D}$ such that 
$$ \aV{e_n - 1}_o \to 0, \text{ as } n \to \infty.$$
Furthermore, let $u\in \mathbf{D}$ with $0 \leq u \leq 1$ be given. Then, $$\aV{e_n \wedge u - u}_o \to 0\text{, as } n \to \infty.$$ 
\end{coro}
\begin{proof}
Proposition \ref{structure of D} shows that convergence with respect to $\aV{\cdot}_o$ implies pointwise convergence, hence $e_n \to 1$ pointwise. Since $0 \leq u \leq 1$ this also yields $e_n\wedge u \to u$ pointwise. Using Theorem \ref{convergence in D} it suffices to show $\limsup \ow{Q}(e_n \wedge u) \leq \ow{Q}(u)$ to infer the statement. By the triangle inequality for $\ow{Q}^{1/2}$ and Lemma \ref{Markov property} we obtain
\begin{align*}
\ow{Q}(e_n \wedge u)^{1/2} &= \frac{1}{2}\ow{Q}(u+e_n-|u-e_n|)^{1/2}\\
&\leq \frac{1}{2}\left[\ow{Q}(u)^{1/2}+ \ow{Q}(e_n)^{1/2}  + \ow{Q}(|u-e_n|)^{1/2}\right]\\
&\leq \ow{Q}(u)^{1/2} + \ow{Q}(e_n)^{1/2}.
\end{align*}
Since $c\equiv 0$ the constant function $1$ satisfies $\ow{Q}(1) = 0$. Hence, the convergence $\aV{e_n - 1}_o \to 0$ implies $\ow{Q}(e_n) \to 0,$ as $n \to \infty$. With this observation the above calculation yields the statement after taking $\limsup$.
\end{proof}


\subsection{Dirichlet forms}

In this subsection we introduce Dirichlet forms associated with graphs and prove some of their properties. For basic definitions of Dirichlet form theory we refer to Appendix~\ref{appendix:dirichlet forms}. Let $m$ be a measure of full support on $V$, i.e., a function $m:V \to (0,\infty)$ which induces a measure by setting
$$m(A) := \sum_{x\in A}m(x).$$
The sets  
$$\ell^p(V,m) := \{u:V \to \RR \,|\,  \sum_{x \in V} |u(x)|^p m(x) < \infty \}$$ 
endowed with the norm 
$$\aV{u}_p := \left( \sumv{x}|u(x)|^p m(x) \right)^{1/p} $$
are Banach spaces. Moreover, the case $p=2$ provides a Hilbert space with inner product given by 
$$\as{f,g} := \sumv{x}f(x)g(x)m(x). $$
For positive $f,g \in C(V)$ we abuse notation and let
$$\as{f,g} := \sum_{x\in V} f(x)g(x)m(x) \in [0,\infty].$$

As usual, $\ell^\infty(V)$ denotes the set of all bounded functions on $V$. It comes with the corresponding norm
$$\aV{u}_\infty := \sup _{x \in V} |u(x)|.$$
In the subsequent chapters we will be concerned with restrictions of $\ow{Q}$ to certain $\ell^2$-domains such that the emerging forms become Dirichlet forms. There is a maximal and a minimal choice for such domains. The maximal $\ell^2$-domain given by $D(\QN) = \ow{D} \cap \ell^2(V,m)$. We  call $\ow{Q}|_{D(\QN)}$ the {\em Neumann form associated to $(b,c)$} and write $\QN$ for short. The following summarizes properties of $\QN$ and its associated operator. 
\begin{prop} \label{QN}
$\QN$ is a Dirichlet form. Its associated self-adjoint operator $\LN$ is a restriction of $\ow{L}$ with domain $D(\LN)$ satisfying
$$D(\LN) \subseteq \{u \in \ow{D}\cap\ell^2(V,m) \,| \, \, \ow{L}u \in \ell^2(V,m)\}.$$ 
\end{prop}
\begin{proof} We first show the closedness of $\QN$. To this end, let $(u_n)$ be a Cauchy sequence  in $D(\QN)$ with respect to the inner product $\QN(\cdot,\cdot) + \as{\cdot,\cdot}$. It has a limit $u \in \ell^2(V,m)$ with respect to $\|\cdot\|_2$. Since $\ell^2$-convergence implies pointwise convergence, Fatou's lemma yields
$$\ow{Q}(u - u_n) \leq \liminf_{m\to \infty} \ow{Q} (u_m - u_n) = \liminf_{m\to \infty} \QN (u_m - u_n).$$
This inequality shows $u\in D(\QN)$ and $u_n \to u$ with respect to $\QN(\cdot,\cdot) + \as{\cdot,\cdot}$, i.e., the closedness of $\QN$.

Next we come to the Markov property of $\QN$. Let $F: \RR \to \RR$ be a normal contraction. Lemma~\ref{Markov property} shows that for $u \in \ow{D}$ the inequality
$$\ow{Q}(F \circ u) \leq \ow{Q}(u)$$
holds. If, furthermore, $u \in \ell^2(V,m)$ the function  $F\circ u$ also belongs to $\ell^2(V,m)$. This proves the Markov property of $\QN$.

For the statement on the associated operator let $u \in D(\LN)$. Using the notation $\hat{\delta}_x = m(x)^{-1}\delta_x$ and Lemma~\ref{green} we obtain
$$(\LN u)(x) = \as{\LN u,\hat{\delta}_x} = \QN(u,\hat{\delta}_x) = \as{\ow{L}u,\hat{\delta}_x} = (\ow{L}u)(x).
$$
This shows that $\LN$ is a restriction of $\ow{L}$ and also implies 
$$D(\LN) \subseteq \{u \in \ow{D}\cap\ell^2(V,m) \,| \, \, \ow{L}u \in \ell^2(V,m)\}.  $$ 
This finishes the proof.
\end{proof}

The second important choice for a domain is given by
$$D(\QD) = \overline{C_c(V)}^{\aV{\cdot}_Q},$$
where the closure is taken with respect to the form norm 
$$\|\cdot\|_Q := \sqrt{\ow{Q}(\cdot) + \aV{\cdot}_2^2},$$
in $\ell^2(V,m)\cap \ow{D}$. We  call $\ow{Q}|_{D(\QD)}$ the {\em regular Dirichlet form associated to $(b,c)$} and denote it by $\QD$. It might be thought of being the minimal closed $\ell^2$-restriction of $\ow{Q}$ containing $C_c(V)$. 
\begin{prop} \label{QD}
$\QD$ is a regular Dirichlet form. Its associated operator $\LD$ is a restriction of $\ow{L}$ with domain $D(\LD)$ satisfying
$$D(\LD) = \{u \in D(\QD) \,| \, \ow{L}u \in \ell^2(V,m)\}.$$
\end{prop}

\begin{proof}
The closedness of $\QD$  follows from the closedness of $\QN$ and the fact that $D(\QD)$ is a closed subspace of $D(\QN)$ with respect to the form norm $\|\cdot\|_Q$. 

For showing the Markov property of $\QD$ we let $F$ be a normal contraction and $u \in D(\QD)$. It suffices to show $F \circ u \in D(\QD)$. By the definition of $D(\QD)$ there exists a sequence $(u_n)$ in $C_c(V)$ converging towards $u$ with respect to $\aV{\cdot}_Q$. Obviously, we have 
$$\aV{F \circ u_n - F\circ u}_2 \to 0\text{ as } n \to \infty.$$

Furthermore,
$$\limsup _{n \to \infty} \ow{Q}(F \circ u_n) \leq \limsup _{n \to \infty} \ow{Q}(u_n) = \ow{Q}(u).$$
Therefore, the sequence $(F \circ u_n)$ is  bounded within the Hilbert space $(D(\QN), \aV{\cdot}_Q)$. By the weak compactness of balls in Hilbert spaces we conclude that any of its subsequences has a weakly convergent subsequence. Because of the $\ell^2$-convergence of $F \circ u_n$ all the occurring limits must coincide with $F \circ u$. We infer $F\circ u_n \to F\circ u$ weakly in $(D(\QN), \aV{\cdot}_Q)$. This shows that $F \circ u$ belongs to the weak closure of $C_c(V)$ in $(D(\QN), \aV{\cdot}_Q)$. Because $C_c(V)$ is convex, this weak closure coincides with the closure of $C_c(V)$ with respect to $\aV{\cdot}_Q$ and we obtain $F \circ u \in D(\QD)$.

We now come to the statement on the operator $\LD$. With the same arguments as in the proof if Proposition \ref{QN} we can show that $\LD$ is a restriction of $\ow{L}$ and  that its domain satisfies
$$D(\LD) \subseteq \{u \in D(\QD)\,| \, \ow{L}u \in \ell^2(V,m)\}.$$
It remains to prove the opposite inclusion. To this end, we let $v \in \{u \in D(\QD) \,| \, \ow{L}u \in \ell^2(V,m)\}$. By the correspondence of $\QD$ and $\LD$ (see Appendix~\ref{appendix:dirichlet forms}) we need to confirm the validity of 
$$\QD(v,w) = \as{\ow{L}v,w}$$
for all $w \in D(\QD)$. From Lemma \ref{green} we infer that the above equality holds true for all $w \in C_c(V)$. Since $C_c(V)$ is dense in $D(\QD)$ with respect to $\aV{\cdot}_Q$,  it extends to all $w \in D(\QD)$. This finishes the proof.
\end{proof}

\begin{remark}
\begin{itemize}
\item The previous proof and the one of Theorem \ref{convergence in D} use a similar argument. They show that a bounded sequence in a Hilbert space that is convergent in some 'weak' sense (pointwise, in $\ell^2$, ...) is already weakly convergent in that Hilbert space.
\item There seems to be no explicit proof for $\QD$ being a Dirichlet form in the literature.  Usually this property is deduced from \cite[Theorem~3.1.1]{FOT} which uses general principles.  
\item The given characterization of the domain of $\LD$ seems to be new. 
\end{itemize}
\end{remark}

We now come to the main objects of our studies.

\begin{definition}[Dirichlet form associated with $(b,c)$]
Let $(b,c)$ be a graph over $V$. A Dirichlet form $Q$ is called associated to $(b,c)$ if it is a restriction of $\ow{Q}_{b,c}$ and its domain $D(Q)$ satisfies
$$ D(\QD) \subseteq D(Q) \subseteq D(\QN).$$
\end{definition}

\begin{remark} The Dirichlet forms considered above seem to be a special class of examples of Dirichlet forms on $\ell^2(V,m)$. However, as seen in \cite{KL}, it turns out that every regular Dirichlet form on a discrete space coincides with a regular Dirichlet form associated to a graph $(b,c)$ (see also Appendix~\ref{subsection:dirichlet forms of graphs}). In this sense the forms $\QD$ are exactly the regular Dirichlet forms on discrete spaces.
\end{remark}




\section{General theory} \label{general theory}

In this section we study recurrence and transience of Dirichlet forms which are associated to a graph $(b,c)$ by using the theory presented in \cite[Section~1.5]{FOT}. The discrete structure of the underlying $\ell^2$-space allows us to simplify many technical details. Therefore, some definitions and statements differ slightly from the ones found in \cite{FOT}. 


Let $Q$ be a Dirichlet form associated to a graph $(b,c)$ and let $e^{-tL}$ be its associated semigroup. Recall that $e^{-tL}$ is positivity preserving, i.e., it maps positive functions to positive functions (see Appendix~\ref{appendix:dirichlet forms}).

\begin{definition}[Transient semigroup]
The semigroup $e^{-tL}$ is called \emph{transient} if 
$$G(x,y) := \int_{0}^{\infty} e^{-tL}\delta_x (y) dt < \infty, \label{Definition transient semigroup}$$ 
for every $x,y \in V$. It is called \emph{recurrent} if $G(x,y) = \infty$ for all $x,y \in V$. The function $G:V \times V \to [0,\infty]$ is the {\em Green's function} of $Q$. 
\end{definition}

The next proposition shows the dichotomy of the above concepts whenever the graph $(b,c)$ is connected.

\begin{prop}\label{dichotomy}
	Let $(b,c)$ be connected. The semigroup $e^{-tL}$ is transient if and only if there exist some $x,y \in V$ such that $G(x,y) < \infty$. In particular, $e^{-tL}$ is either recurrent or transient. 
\end{prop}

\begin{proof}
Let $x,y \in V$ such that $G(x,y) < \infty$ and let $w,z\in V$ be  arbitrary. We need to show $G(w,z) < \infty$. The functions $\ow{\delta}_v = m(v)^{-1/2}\delta_v$ form an orthonormal basis in $\ell^2(V,m)$. Using the semigroup property and that $e^{-tL}$ is positivity preserving (see Appendix~\ref{appendix:dirichlet forms}) we then obtain for $t > 1$
\begin{align}
	(e^{-tL}\delta_x)(y) &= (e^{-(t-1) L}e^{- L}\delta_x)(y) \notag\\
	&= \left[e^{-(t-1) L}\left(\sum_{v \in V} \as{e^{- L}\delta_x, \ow{\delta}_v}\ow{\delta}_v\right)\right](y)\notag\\
	&=\frac{1}{m(y)}\sum_{v \in V} \as{e^{- L}\delta_x, \ow{\delta}_v}\as{e^{-(t-1) L}\ow{\delta}_v,\delta_y}\notag\\
	&\geq \frac{1}{m(y)}\as{e^{- L}\delta_x, \ow{\delta}_w}\as{e^{-(t-1) L}\ow{\delta}_w,\delta_y}. \label{inequality 1}
\end{align}

Since $(b,c)$ is connected, Theorem \ref{positivity improving} shows that $e^{-tL}$ is positivity improving, hence
$$\as{e^{- L}\delta_x, \ow{\delta}_w} > 0.$$
We can now conclude the finiteness of $G(w,y)$ by integrating both sides of inequality \eqref{inequality 1} from $1$ to $\infty$. A similar computation shows the finiteness of $G(w,z)$. Now the 'In particular' part is an immediate consequence of the previous. 
\end{proof}
The next proposition shows how transience of the semigroup is related to the resolvent $(L+\alpha)^{-1}$ associated with $Q$.
\begin{prop} \label{correspondence resolvent G}
For all $x,y \in V$ the equality
$$G(x,y) = \int_{0}^{\infty} e^{-tL}\delta_x (y) dt = \lim_{\alpha \to 0+} (L+\alpha)^{-1}\delta_x (y) $$
holds. 
\end{prop}
\begin{proof}
By monotone convergence we infer
$$\int_{0}^{\infty} e^{-tL}\delta_x (y) dt = \lim_{\alpha \to 0+}\int_{0}^{\infty} e^{-t\alpha}e^{-tL}\delta_x (y) dt.$$
Let $\mu_{x,y}$ be the spectral measure associated with $L$ such that 
$$\frac{1}{m(y)}\as{e^{-tL}\delta_x,\delta_y} = \int_0^\infty e^{-t\lambda}d\mu_{x,y} (\lambda).$$
Since $e^{-t(\alpha + \lambda)}$ is integrable in $t$ on $[0,\infty)$ and $\mu_{x,y}$ is of bounded total variation, we can use Fubini's theorem to obtain
\begin{align*}
\int_{0}^{\infty} e^{-t\alpha}e^{-tL}\delta_x (y) dt &= \int_{0}^{\infty}\int_0^\infty e^{-t\alpha} e^{-t\lambda}d\mu_{x,y} (\lambda) dt\\
&=  \int_{0}^{\infty}\int_0^\infty e^{-t(\alpha+\lambda)} dt d\mu_{x,y} (\lambda) \\
&=  \int_{0}^{\infty} \frac{1}{\lambda + \alpha} d\mu_{x,y}(\lambda)\\
&= (L + \alpha)^{-1}\delta_x (y).
\end{align*}
This finishes the proof.
\end{proof}

There is also a notion of recurrence and transience for Dirichlet forms. This is discussed next. 

\begin{definition}[Transient Dirichlet form]
The Dirichlet form $Q$ is called \emph{transient} if there exists a strictly positive $g \in \ell^1(V,m)\cap \ell^{\infty}(V)$ such that the inequality
\[ \sum_{x\in V} |u(x)|g(x)m(x) \leq \sqrt{Q(u)}, \label{Definition transient space} \] 
holds for all $u\in D(Q)$. Such a function $g$ is called \emph{reference function} for $Q$.  
\end{definition}
The semigroup $(e^{-tL})$ is a strongly continuous family of bounded operators on $\ell^2(V,m)$. Thus, the real function $t \mapsto e^{-tL}f(x)$ is continuous for any $f\in \ell^2(V,m)$ and any $x \in V$. We abuse notation and introduce the {\em 0-th order resolvent operator} 
$$G: \ell^{2}_+ (V,m) \to \{u:V \to [0,\infty] \}$$
acting by
$$ (Gf)(x) =  \int_{0}^{\infty} e^{-tL}f (x) dt.$$
Here $\ell^2_+(V,m)$ denotes the set of all non-negative $\ell^2$-functions. Note that $(Gf)(x)$ may take the value $\infty$. 
\begin{remark} \label{remark:green functiona and resolvent}
 It is immediate from the definition that the Green's function of $Q$ satisfies $G(x,y) = (G\delta_x)(y)$. In this sense it is the integral kernel of the 0-th order resolvent operator with respect to the counting measure.  
\end{remark}
 The following lemma will allow us to prove the equivalence of Definition~\ref{Definition transient semigroup} and of Definition~\ref{Definition transient space}, it is a variant of \cite[Lemma~1.5.3]{FOT}. 
\begin{lemma} \label{Main} Let $g\in \ell^1(V,m)\cap \ell^2(V,m)$ be nonnegative. Then,
$$\sup_{u\in D(Q)} \frac{\as{|u|,g }}{\sqrt{Q(u)}} =  \sqrt{\as{g,Gg}}.$$
In particular,
$$\sup_{u\in D(Q)} \frac{\as{|u|,g }}{\sqrt{Q(u)}} < \infty \text{ if and only if }  \as{g,Gg} < \infty.$$ 
\end{lemma}

\begin{proof}
For $f \in \ell^2(V,m)$ and $x \in V$ we use the notation
$$(S_tf)(x) := \int_0^t e^{-sL}f(x)ds.$$
With the help of Theorem \ref{integral inequality} we conclude 
$$\aV{S_tf}_2 \leq \int_0^t \aV{e^{-sL}f}_2ds \leq t \aV{f}_2.$$
Therefore, $S_tf$ is a bounded linear operator on $\ell^2(V,m)$. The proof will be done in three steps.

\emph{Step 1}: For $f \in \ell^2(V,m)$ and $h \in D(Q)$ we prove the identity
\begin{align} 
	Q(S_tf, h) = \as{f-e^{-tL}f,h} \label{Formeq}
\end{align}
by showing that $S_tf \in D(L)$ and $LS_tf = f-e^{-tL}f$. Using the correspondence of $L$ and  $e^{-tL}$ we need to compute the derivative of the function $[0,\infty) \to \ell^2(V,m)$, $s \mapsto -e^{-sL}S_tf$ at 0 (see Appendix~\ref{appendix:dirichlet forms}). Theorem \ref{vector integration bounded operator} yields
\begin{align*}
 \frac{S_tf-e^{-sL}S_tf}{s} &= \frac{S_tf-S_te^{-sL}f}{s}\\
 &= \frac{1}{s}\left(\int_0^s e^{-uL}f  du - \int_{t}^{s+t}e^{-uL}fdu\right).
\end{align*}
With the help of Theorem \ref{integral inequality} and a mean value theorem for Riemann integrals we compute
\begin{align*}
\aV{\frac{1}{s}\int_0^s e^{-uL}f  du - f}_2 &= \aV{\frac{1}{s}\int_0^s e^{-uL}f-f  du }_2\\
&\leq \frac{1}{s}\int_0^s \aV{e^{-uL}f-f}_2  du \\
&= \aV{e^{-\theta L}f-f}_2
\end{align*}
where $\theta \in (0,s).$ Taking the limit $s \to 0$  yields
$$\frac{1}{s}\int_0^s e^{-uL}f  du \to  f.$$
An analogous computation shows 
$$\frac{1}{s}\int_{t}^{s+t}e^{-uL}fdu \to e^{-tL}f \text{ as } s \to 0.$$
This implies Equation~\eqref{Formeq}.

\emph{Step 2}: We show that
$$c:= \sup_{u\in D(Q)} \frac{\as{|u|,g}}{\sqrt{Q(u)}} < \infty.$$ 
implies $\sqrt{\as{g,Gg}} \leq c$.  Since $e^{-tL}g\geq 0$ and $S_tg \geq 0$ we obtain
$$
\as{g,S_tg} \leq c\sqrt{Q(S_tg)} \overset{\eqref{Formeq}}{=} c \sqrt{\as{g-e^{-tL}g,S_tg}} \leq c\sqrt{\as{g,S_tg}}.$$
We conclude $\sqrt{\as{S_t g,g}} \leq c$. Letting $t \to \infty$ we obtain $\sqrt{\as{g,Gg}} \leq c$ by the monotone convergence theorem.

\emph{Step 3}: Suppose that $\as{g,Gg} < \infty$. By monotone convergence and semigroup properties, we conclude 
\begin{align*}
\as{g,Gg}	&= \int_0^{\infty} \as{e^{-t L}g,g} dt\\
	&= \int_0^{\infty} \as{e^{-t/2 L}g,e^{-t/2 L}g} dt.
\end{align*}
The function  $ t \mapsto \as{e^{-t/2 L}g,e^{-t/2 L}g}$ is nonincreasing in $t$ (use the semigroup property and  $\aV{e^{-sL}}\leq 1$). Thus, the finiteness of $\as{g,Gg}$ yields 
 $$ \lim_{t\to \infty} \as{e^{-t L}g,g} =  \lim_{t\to \infty} \as{e^{-t/2 L}g,e^{-t/2 L}g}= 0.$$
On account of \eqref{Formeq} for $u \in D(Q)$ we obtain 
\begin{align*}
	\as{|u|,g} &\overset{\eqref{Formeq}}{=} Q(|u|,S_tg) + \as{|u|,e^{-tL}g} \\
	&\leq \sqrt{Q(S_tg)}\sqrt{Q(|u|)} + \sqrt{\as{e^{-tL}g,e^{-tL}g}} \sqrt{\as{u,u}} \\
	&\overset{\eqref{Formeq}}{=} \sqrt{\as{g-e^{-tL}g,S_tg}}\sqrt{Q(|u|)} + \sqrt{\as{e^{-tL}g,e^{-tL}g}} \sqrt{\as{u,u}}\\
	&\leq \sqrt{\as{S_tg,g}}\sqrt{Q(u)} + \sqrt{\as{e^{-2tL}g,g}} \sqrt{\as{u,u}}.
\end{align*}
The inequality $\as{|u|,g} \leq \sqrt{\as{g,Gg}}\sqrt{Q(u)}$ follows by taking the limit $t\to \infty$ in the above computations. 
\end{proof}

The next theorem shows that transience of the Dirichlet form $Q$ and of its associated semigroup coincide. It is a variant of \cite[Theorem 1.5.1]{FOT}.

\begin{theorem} \label{Transience}
The Dirichlet form $Q$ is transient if and only if its associated semigroup $e^{-tL}$ is transient.
\end{theorem}
\begin{proof}
Assume $Q$ is transient and let $g \in \ell^1(V,m) \cap \ell^\infty(V) \subseteq \ell^2(V,m)$ be a reference function of $Q$. On account of 
$$\sum_{x \in V} |u(x)|g(x)m(x) \leq \sqrt{Q(u)} \text{ for } u \in D(Q),$$
Lemma \ref{Main} shows that $\as{g,Gg} \leq 1$ must hold. Since reference functions are strictly positive this implies $Gg(x) < \infty$ for all $x\in V$. Therefore, the non-negativity of $e^{-tL}\delta_x$ and the self-adjointness of $e^{-tL}$ yield
\begin{align*}
	\int_0^\infty e^{-tL}\delta_x (y) dt &= \frac{1}{g(y)m(y) }\int_0^\infty   \as{e^{-tL}\delta_x, g(y)\delta_y} dt \\
	&\leq \frac{1}{g(y)m(y) }\int_0^\infty   \as{e^{-tL}\delta_x, g} dt\\
	&= \frac{m(x)}{g(y)m(y) }\int_0^\infty   e^{-tL} g(x) dt < \infty.
\end{align*}

Now assume that $e^{-tL}$ is transient, i.e., its Green's function satisfies $G(x,y) < \infty$ for all $x,y \in V$. By Remark~\ref{remark:green functiona and resolvent} the equality $$G(x,y) = G\delta_x (y)$$ holds for all $x,y \in V$.  Hence, an application of Lemma \ref{Main} to $\delta_x$ yields
$$
	\sup_{u\in D(Q)} \frac{|u(x)|m(x)}{\sqrt{Q(u)}} = G(x,x)m(x). 
$$
Since $D(Q)$ is dense in $\ell^2(V,m)$, for every $x \in V$ there exists a function $v_x \in D(Q)$ such that $|v_x(x)| >0$. Therefore the above equation shows $G(x,x) > 0$ for every $x \in V$. We need to find a reference function $g$ as in Definition \ref{Definition transient space}. Let us define $g$ as
$$
	g(x) = \frac{a_x}{G(x,x)m(x)},
$$
where $(a_x)$ is a sequence of strictly positive numbers, chosen such that $g$ belongs to $\ell^1(V,m)\cap \ell^{\infty}(V)$ and \[\sum_{x\in V} a_x = 1.\]
We obtain for $u\in D(Q)$
\begin{align*}
\sum_{x \in V} |u(x)|g(x)m(x) &=  \sum_{x \in V}  \frac{|u(x)|a_x}{G(x,x)}\\
& \leq \sum_{x\in V} a_x \sqrt{Q(u)} \\
&=\sqrt{Q(u)},
\end{align*}
as was to be shown.
\end{proof}

As a next step we introduce the extended Dirichlet space $D(Q)_e$ associated with $Q$. It will turn out that properties of $D(Q)_e$ are related to transience of $Q$. 

\begin{definition}[Extended Dirichlet space] \label{extended space} The set $D(Q)_e$ 
\begin{align*}
D(Q)_e := \{u: V \to \RR\,|\, &\text{ there exists a } Q\text{-Cauchy sequence }\\ &(u_n) \text{ in } D(Q)   \text{ s.t. } u_n \to u \text{ pointwise} \},
\end{align*}
is called the \emph{extended Dirichlet space} of $Q$. A sequence $(u_n)$ as in the defintion of $D(Q)_e$ set is an \emph{approximating sequence} for $u$.
\end{definition}

We extend the Dirichlet form $Q$ to its extended space. This can be done via the next lemma.

\begin{lemma} \label{extended form}
Let $u \in D(Q)_e$ and let $(u_n)$ in $D(Q)$ be an approximating sequence for $u$. Then, $u \in \ow{D}$ and
$$\lim_{n \to \infty}\ow{Q}(u-u_n) = 0.$$
\end{lemma}
\begin{proof}
By Fatou's lemma we obtain
\begin{align*}
	\ow{Q}(u-u_n)^{1/2} 	&\leq \liminf_{l\to \infty} \ow{Q}(u_l - u_n)^{1/2}\\
	& = \liminf_{l\to \infty} Q(u_l - u_n)^{1/2}. 
\end{align*}
Since $(u_n)$ is a $Q$-Cauchy sequence, this yields $u \in \ow{D}$ and $\lim\ow{Q}(u-u_n) = 0.$
\end{proof}
For $u,v \in D(Q)_e$ we set 
$$Q(u,v) := \ow{Q}(u,v) $$
to extend $Q$ to $D(Q)_e$. The previous lemma shows that $D(Q)$ is dense in $D(Q)_e$ with respect to the pseudometric induced by $Q$ on $D(Q)_e$.  Whenever the underlying graph $(b,c)$ is connected $D(Q)_e$ can be computed as follows. 

\begin{prop} \label{char extended space}
Let $(b,c)$ be connected. The extended Dirichlet space of a Dirichlet form $Q$  associated to $(b,c)$ is given by the closure of $D(Q)$ in $\mathbf{D}$, i.e.,
$$D(Q)_e  = \overline{D(Q)}^{\aV{\cdot}_o}.$$ 
\end{prop}
\begin{proof}
Let $u \in D(Q)_e$ be given and $(u_n)\subseteq D(Q)$ be an approximating sequence for $u$. Then $(u_n)$ is a Cauchy sequence with respect to $\aV{\cdot}_o$. Since $\mathbf{D}$ is complete $(u_n)$ converges in $\mathbf{D}$ to some function $v$. By the pointwise convergence of the $u_n$ to $u$ we infer $u = v$ . This shows $u \in \overline{D(Q)}^{\aV{\cdot}_o}$. The other inclusion follows from Proposition \ref{structure of D}.
\end{proof}
\begin{coro} \label{extended Dirichlet space regular DF}
 Let $(b,c)$ be connected and $\QD$ be the regular Dirichlet form associated with $(b,c)$ on $\ell^2(X,m)$. Its extended Dirichlet space is given by
 $$D(\QD)_e = \mathbf{D}_0.$$
\end{coro}
\begin{proof}
This is a consequence of the following equation
$$D(Q)_e = \overline{D(\QD)}^{\aV{\cdot}_o} = \overline{\overline{C_c(V)}^{\aV{\cdot}_Q}}^{\aV{\cdot}_o} =  \overline{C_c(V)}^{\aV{\cdot}_o},$$
where the first equality follows from the previous proposition and the last equality follows from the fact that the norm $\aV{\cdot}_o$ is continuous with respect to $\aV{\cdot}_Q = (\ow{Q}(\cdot) + \aV{\cdot}^2)^{1/2}$. 
\end{proof}

\begin{remark}
That the extended Dirichlet space of $Q$ is given as the closure of $D(Q)$ in $\mathbf{D}$ seems to be a new observation. 
\end{remark}

We turn our studies towards criteria for transience. First let us observe a necessary condition. 

\begin{lemma} \label{lemma:etended space transient form} Suppose $Q$ is transient. Then, $(D(Q)_e,Q)$ is a Hilbert space.  
\end{lemma}
\begin{proof}
By transience there exists a strictly positive $g \in \ell^1(V,m)\cap \ell^\infty(V)$ such that
\begin{gather}
	 \as{|u|,g} \leq \sqrt{Q(u)}, \label{inequality 2}
\end{gather}
for all $u \in D(Q)$. By the definition of $D(Q)_e$ and Lemma \ref{extended form} this inequality extends to all $u \in D(Q)_e$, thus $Q$  is non-degenerate on $D(Q)_e$. For proving completeness let $(u_n)$ be a Cauchy sequence in $(D(Q)_e, Q)$. Inequality \eqref{inequality 2} implies the pointwise convergence of $(u_n)$ to a function $u$. This yields $u \in D(Q)_e$ with approximating sequence $(u_n)$. We then infer from Lemma \ref{extended form}
\begin{align*}
Q(u_n - u) \to 0 \text{, as } n \to \infty.
\end{align*} 
This finishes the proof.
\end{proof}

Below we prove characterizations for recurrence and transience. In the recurrent case an important step will be to construct a sequence of functions in $D(Q)$ converging to $1$ with respect to $Q$. To obtain such a sequence we will need to study a transient Dirichlet form $Q^g$, which is defined by perturbing $Q$ in the following way. Let $g \in \ell^1(V,m)\cap \ell^{\infty}(V)$ be strictly positive. We then set
 \begin{gather*}
 Q^g: D(Q) \times D(Q) \to \RR \\
 Q^g(u,v) = Q(u,v) + \as{gu,v}. 
 \end{gather*}
The following proposition shows that $Q^g$ is indeed a transient Dirichlet form. 

\begin{prop}
$Q^g$ is a transient Dirichlet form.
\end{prop}
\begin{proof}
Obviously the form $Q^g$ has the Markov property. As the sum of a closed and a continuous form it is also closed. 
For showing its transience we compute 
\[ \sum_{x\in V} |u(x)|g(x)m(x) \leq \sqrt{\|g\|_1}\sqrt{\as{gu,u}} \leq \sqrt{\|g\|_1} \sqrt{Q^g(u)}. \] 
Thus, we can choose $\tilde{g} = g/\sqrt{\|g\|_1}$ as a reference function.
\end{proof}

Let $L^g$ denote the self-adjoint operator which is associated with $Q^g$.  The following lemma is the key ingredient for characterizing recurrence. It is a variant of \cite[Lemma 1.6.6]{FOT}.

\begin{lemma} \label{main help lemma for recurrence}
 Let $Q$ be recurrent and let $g$ be as above. For any $x \in V$
 $$\lim_{\alpha \to 0+} (L^g + \alpha)^{-1}g(x) = 1.$$
\end{lemma}
\begin{proof}
 We write $Q_\alpha$ for the perturbed form
$$Q_\alpha(u,v) := Q(u,v) + \alpha \as{u,v}.$$
Recall that  
$$Q_\alpha(w,u) = \as{f,u}$$
holds for all $u \in D(Q)$ if and only if  $w = (L+\alpha)^{-1}f$ (see Appendix~\ref{appendix:dirichlet forms}). The proof will be done in three steps.

\emph{Step 1}: We compute the resolvent of $Q^g$ in terms of the resolvent of $Q$. Using the above characterization of the resolvent with respect to  $Q_\alpha$ let us observe that for any $f\in \ell^2(V,m)$ and $u \in D(Q)$
\begin{align*}
	 Q_\alpha \left((L^g + \alpha)^{-1}f,u \right) &= Q_\alpha^g \left((L^g + \alpha)^{-1}f,u\right) - \as{g(L^g + \alpha)^{-1}f,u}\\
	 &= \as{f - g(L^g + \alpha)^{-1}f,u}.
\end{align*}
This implies 
\begin{gather} (L^g + \alpha)^{-1}f = (L + \alpha)^{-1}\left[f-g(L^g + \alpha)^{-1}f \right]. \label{resolventeq}
\end{gather}

\emph{Step 2}: We show $(L^g+\alpha)^{-1}g \leq 1$ by considering $Q$ as a Dirichlet form on the space $\ell^2(V, (g+\alpha)m)$. 

Since we assumed $g$ to be positive and bounded the measures $(g+\alpha)m$ and $m$ are equivalent. Thus, $Q$ is closed on $\ell^2(V, (g+\alpha)m)$.  For $f \in \ell^2(V, (g+\alpha)m)$and $v \in D(Q)$ we compute
\begin{align*}
	&Q\left((L^g+ \alpha)^{-1}[\alpha f + fg],v\right) + \as{(L^g+ \alpha)^{-1}[\alpha f + fg],v}_{(g+\alpha)m}\\
	&= Q^g\left((L^g+ \alpha)^{-1}[\alpha f + fg],v\right) + \alpha \as{(L^g+ \alpha)^{-1}[\alpha f + fg],v}\\	
	&=  Q^g_\alpha \left((L^g+ \alpha)^{-1}[\alpha f + fg],v\right)\\
	&= \as{\alpha f + gf, v} = \as{f,v}_{(\alpha+g)m}.
\end{align*}
From this calculation we infer that $(L^g+ \alpha)^{-1}[\alpha f + fg]$ is the 1st-order resolvent of $f$ associated to $Q$ as a Dirichlet form on  $\ell^2(V, (g+\alpha)m)$. The Markov property of this resolvent (see Appendix~\ref{appendix:dirichlet forms}) implies that for any $f \in \ell^2(V,(g+\alpha)m)$ with $0 \leq f \leq 1$ the inequality
\begin{gather}(L^g+ \alpha)^{-1}[\alpha f + fg] \leq 1 \label{inequality3}\end{gather}
holds. Pick a sequence $(f_n)$ in $\ell^2(V,(g+\alpha)m)$ such that $ 0\leq f_n \leq f_{n+1}\leq 1$ and $\lim f_n(x) = 1$ for all $x \in V$. Then $f_n g \to g$ in $\ell^2(V,m)$. Inequality \eqref{inequality3} together with the non-negativity of $(L^g+ \alpha)^{-1}(\alpha f_n)$ yields
\begin{align*}
	(L^g+\alpha)^{-1}g(x) &=  \lim_{n \to \infty} (L^g+\alpha)^{-1}(f_ng)(x)\\
	& \leq  \limsup_{n\to \infty} [1 - (L^g+ \alpha)^{-1}(\alpha f_n)(x)] \\ 
	&\leq 1.
\end{align*}

\emph{Step 3}: Equation \eqref{resolventeq} applied to $g$ and  Step 2 imply
\begin{align*}
	1 &\geq \limsup_{\alpha \to 0+} (L^g+\alpha)^{-1}g(x) \\
	&=  \limsup_{\alpha \to 0+} (L + \alpha)^{-1}\left[g-g(L^g + \alpha)^{-1}g  \right](x) \\
	&= \limsup_{\alpha \to 0+}\frac{1}{m(x)} \as{ g(1-(L^g + \alpha)^{-1}g) ,(L + \alpha)^{-1} \delta_x } \\
	&\geq \limsup_{\alpha \to 0+} \frac{g(x)(1-(L^g + \alpha)^{-1}g(x))}{m(x)} \as{ \delta_x ,(L + \alpha)^{-1} \delta_x } \geq 0.
\end{align*} 
Since $(L + \alpha)^{-1} \delta_x(x)$ converges to infinity by our assumptions ($Q$ was supposed to be recurrent), $(L^g + \alpha)^{-1}g(x)$ must tend to one. This finishes the proof.
\end{proof}

We can now prove the main theorems of this chapter. For general Dirichlet forms they can be found in \cite[Theorem~1.6.2 and Theorem~1.6.3]{FOT}.

\begin{theorem}[Abstract characterization of transience] \label{char transience}
Let $(b,c)$ be connected and $Q$ a Dirichlet form associated with $(b,c)$. Then, the following conditions are equivalent: 
\begin{itemize}
\setlength{\itemsep}{0pt}
\item[(i)]$e^{-tL}$ is transient.
\item[(ii)] $u = 0$ for every $u\in  D(Q)_e$ with $Q(u) = 0$. 
\item[(iii)] $(D(Q)_e,Q)$ is a Hilbert space.
\end{itemize}
\end{theorem}

\begin{proof}
$'(i) \Rightarrow (iii)'$: This has already been shown in Lemma~\ref{lemma:etended space transient form}.

$'(iii)\Rightarrow (ii)'$: This is trivial.

$'(ii) \Rightarrow (i)'$: Assume $e^{-tL}$ is recurrent. Then, for a strictly positive $g \in \ell^1(V,m) \cap \ell^\infty(V)$, Lemma \ref{main help lemma for recurrence} implies $ u_n:= (L^g+1/n)^{-1}g \to 1$ pointwise. Furthermore, by the correspondence $(L^g + \alpha)^{-1} \leftrightarrow Q^g$ we obtain
\begin{align*}
Q(u_n,u_n) &= \as{g,u_n} - \frac{1}{n}\as{u_n,u_n}- \as{gu_n,u_n}.
\end{align*}
Because $g \in \ell^1(V,m)$ and $u_n$ is uniformly bounded by $1$ (see step 2 of the previous proof), we infer with Lebesgue's theorem 
$$ \lim_{n\to \infty}\as{g,u_n} = \lim_{n\to \infty}\as{gu_n,u_n} = \sum_{x \in V}g(x)m(x).  $$
This shows 
$$ Q(u_n,u_n) \leq \as{g,u_n} - \as{gu_n,u_n} \to 0 \text{, as } n\to \infty.$$
The above computations imply $1 \in D(Q)_e$ and Lemma \ref{extended form} yields $Q(1) = 0$. This is a contradiction to $(ii)$ which shows that $e^{-tL}$ is not recurrent. Since $(b,c)$ is connected, this implies transience.
\end{proof}

\begin{theorem}[Abstract characterization of recurrence]\label{char recurrence}
Let $(b,c)$ be connected and $Q$ a Dirichlet form associated with $(b,c)$. Then, the following conditions are equivalent: 
\begin{itemize}
\setlength{\itemsep}{0pt}
\item[(i)]$ e^{-tL}$ is recurrent.
\item[(ii)] There exists a sequence $(u_n)$ in $D(Q)$ satisfying $\lim u_n = 1$ pointwise and $\lim Q(u_n) = 0$.
\item[(iii)] $1 \in D(Q)_e$ and $Q(1) = 0$.
\end{itemize}
\end{theorem}

\begin{proof}
The implication '(i) $\Rightarrow$ (ii)' has already been shown in the proof of Theorem \ref{char transience}.

 '(ii) $\Rightarrow$ (iii)' follows from the definition of $D(Q)_e$ and Lemma \ref{extended form}.

 '(iii) $\Rightarrow$ (i)': By the second statement of Theorem \ref{char transience}, $(iii)$ implies non-transience of $e^{-tL}$. Connectedness of $(b,c)$ implies recurrence.    
\end{proof}

As a consequence of these theorems, let us note that a nonvanishing potential $c$ implies transience. 

\begin{coro} 
Let $(b,c)$ be a connected graph such that $c \not \equiv 0$. Assume $Q$ is a Dirichlet form associated with $(b,c)$. Then, $Q$ is transient. 
\end{coro}
\begin{proof}
If $1$ belonged to $D(Q)_e$, we would obtain 
$$Q(1) = \ow{Q}(1) = \sum_{x\in V} c(x) \neq 0.$$
Then, Theorem \ref{char recurrence} implies that $Q$ is not recurrent since condition $(iii)$ would fail. Because $(b,c)$ is connected, we infer the transience of $Q$.
\end{proof}

\begin{remark}
The previous result is certainly well-known to experts. Because of it will assume $c \equiv 0$ in the subsequent text whenever we deal with recurrence/transience.
\end{remark}

When dealing with the regular Dirichlet form associated to a graph the previous theorems show that recurrence does not depend on the underlying measure $m$. Namely, the following holds.

\begin{coro} \label{Independence of underlying measure}
 Let $(b,c)$ be connected. Let $m_1$ and $m_2$ be two measures of full support on $V$. Let $\QD_{m_i}$ be the regular Dirichlet form associated with $(b,c)$ on $\ell^2(V,m_i)$, $i = 1,2$. Then, $\QD_{m_1}$ is recurrent if and only if $\QD_{m_2}$ is recurrent.  
\end{coro}
\begin{proof}
 We have seen in Corollary~\ref{extended Dirichlet space regular DF} that the extended Dirichlet space of $\QD_{m_1}$ and $\QD_{m_2}$ is given by $\mathbf{D}_0$. As $\mathbf{D}_0$ does not depend on the underlying measure and the previous theorems show that recurrence can be characterized in terms of the extended Dirichlet space the claim follows. 
\end{proof}

\begin{remark} 
 The previous corollary is certainly well-knwon to experts. Nevertheless, it is quite remarkable as the domain of $\QD$ heavily depends on the measure $m$. In fact, the global properties studied in Section \ref{further global properties} also on the underlying measure.
\end{remark}


\section{Discrete time v.s. continuous time} \label{discrete time versus continuous time}

In this section we are going to compare the notion of recurrence of a discrete time Markov chain which is associated with a graph with the one developed in the previous chapter. This notion is usually used in textbooks for the definition of a recurrent graph, see e.g. \cite{Soa,Woe}.


Let $(\Omega, \mathcal{A}, \mathbb{P})$ be a probability space. A sequence of random variables $(X_n)_{n\geq 0}$ on $\Omega$ with values in $V$ will be called a random walk associated with $(b,0)$ if the following conditions are satisfied:
$$
\mathbb{P}(X_{n+1} = y \, | \, X_n = x) = \frac{b(x,y)}{\text{deg}(x)} \text{ for all } x,y \in V, n\geq 0 
$$
and 
\begin{align*}
\mathbb{P}(X_{n+1} = x_{n+1}\,|\, X_n = x_n, X_{n-1} &= x_{n-1},\ldots,X_0 = x_0) \\ &= \mathbb{P}(X_{n+1} = x_{n+1}\,|\, X_n = x_n)
\end{align*}
for all $n \geq 0$ and all $x_0,\ldots,x_{n+1} \in V.$ We skip the proof of the existence of such a Markov chain since this is standard.

Given a random walk associated with $(b,0)$ one might ask about the long time behavior of $(X_n)_{n\geq 0}$. In particular the question whether $(X_n)_{n\geq 0}$ returns to a particular point infinitely often is of interest. 

\begin{definition}[Recurrence random walk] \label{recurrence random walk}
A random walk $(X_n)_{n \geq 0}$ associated to $(b,0)$ is called \emph{recurrent} if  
$$\mathbb{P}(X_n = y \text{ infinitely often } |X_0 = x) = 1$$
for all $x,y \in V$. It is called \emph{transient} if for all $x,y \in V$ the above probability is strictly less than $1$.
\end{definition}

\begin{remark}
 Usually one calls a graph $(b,0)$ recurrent whenever the random walk associated with it is recurrent, see e.g. \cite{Soa,Woe}. The known criteria for recurrence and transience are usually proven in this context.
\end{remark}

Let $P$ be the transition matrix of the random walk associated with $(b,0)$, i.e. the infinite matrix with entries 
$$P(x,y) = \frac{b(x,y)}{\text{deg}(x)}.$$
The following theorem relates recurrence of the random walk of a graph to the matrix $P$. It is a standard exercise in Markov chain theory, see e.g. \cite{FE}.
\begin{theorem}\label{char recurrene random walk}                                                                                                                                                                      
	Let $(X_n)_{n \geq 0}$ be a random walk associated with a connected graph $(b,0)$. Then, $(X_n)_{n\geq 0}$ is recurrent if and only if 
	$$\sum_{n = 0}^{\infty} P^{(n)}(x,y) = \infty$$
	for all $x,y \in V$. Here $ P^{(n)}$ denotes powers of the matrix $P$.
\end{theorem}

We will now discuss the relation of recurrence of Dirichlet forms and recurrence of the random walk associated to a weighted graph $(b,0)$. Assume $m \equiv 1$ and let us write $\widetilde{L} = D - A$, where $D$ and $A$ are the two infinite matrices with entries given by 
$$D(x,y) = \begin{cases}  \text{deg}(x) &\text{if } x = y \\ 0 &\text{if } x \neq y  \end{cases},$$ 
and
$$A(x,y) = \begin{cases}  0 &\text{if } x = y \\ b(x,y) &\text{if } x \neq y  \end{cases}.$$
 Then $P = D^{-1}A$ is the transition matrix of the random walk associated to $(b,0)$ (here $D^{-1}$ is the formal inverse of $D$, i.e., an infinite matrix having $\text{deg}^{-1}$ on its diagonal). 
 
\begin{theorem} \label{discrete time v.s. continuous time}
Let $(b,0)$ be connected and let $e^{-t\LD}$ be the semigroup associated with $\QD$ on $\ell^2(V,1)$. For any $x,y \in V$ the equality
$$G(x,y) = \int_0^\infty e^{-t\LD}\delta_x(y)dt = \frac{1}{\deg(x)}\sum_{n=0}^{\infty}P^{(n)}(x,y)$$
holds.
\end{theorem} 
\begin{proof}
Let $x,y \in V$. By Proposition \ref{correspondence resolvent G} we  obtain
$$\int_0^\infty e^{-t\LD}\delta_x(y)dt = \lim_{\alpha \to 0+} (\LD+\alpha)^{-1}\delta_x(y).$$
Pick an exhaustion $(K_i)$ of $V$ (i.e. $K_i \subseteq K_{i+1}$ and $\cup K_i = V$) with finite $K_i$  and  $x \in K_1$.  Theorem \ref{approximation resolvent} shows that the resolvent of $\QD$ can be approximated by 
$$(\LD+\alpha)^{-1}\delta_x = \lim_{i\to \infty} (L_{K_i}+\alpha)^{-1}\delta_x, $$
where $L_{K_i}$ is the restriction of $\widetilde{L}$ to $C(K_i)$ (in the sense described in Theorem~\ref{approximation resolvent}). We now compute the right-hand side of the previous equation. 

By  $A_{K_i}, D_{K_i}$ we  denote the restrictions of $A, D$ to $C(K_i)$ in the sense used in Theorem \ref{approximation resolvent}, i.e., $A_{K_i} := p_{K_i} A i_{K_i}$ and $D_{K_i} :=  p_{K_i} A i_{K_i}$, where $i_{K_i}: C(K_i) \to C(V)$ is the canonical inclusion and $p_{K_i}: C(V) \to C(K_i)$ is the restriction of a function to $K_i$. With this notation we obtain the matrix identity
\begin{align}
L_{K_i} + \alpha = D_{K_i}  - A_{K_i} + \alpha =  (D_{K_i} + \alpha)(I - (D_{K_i}+\alpha)^{-1}A_{K_i}) \label{equation:Matrix identity}
\end{align}
on the finite-dimensional space $C(K_i)$. The matrix $(I - (D_{K_i}+\alpha)^{-1}A_{K_i})$ is invertible since the operator norm of $(D_{K_i}+\alpha)^{-1}A_{K_i}$ considered as an operator on $\ell^{\infty} (K_i)$ is strictly less than $1$. To see this we let $f \in \ell^{\infty} (K_i)$ with $\|f\|_\infty \leq 1$ and observe
\begin{align*}
\|(D_{K_i}+\alpha)^{-1}A_{K_i}f\|_\infty &\leq \max_{x \in K_i} \frac{1}{\text{deg}(x)+\alpha} \sum_{y\in K_i} b(x,y)|f(y)| \\
& \leq \max_{x \in K_i} \frac{\text{deg}(x)}{\text{deg}(x)+\alpha} < 1.
\end{align*}

Inverting both sides of Equation~\eqref{equation:Matrix identity} and using the von Neumann series expansion for the inverse of $(I - (D_{K_i}+\alpha)^{-1}A_{K_i})$ yields for $y \in K_i$ 
\begin{align}
(L_{K_i}+\alpha)^{-1}\delta_x(y) &= (I - (D_{K_i}+\alpha)^{-1}A_{K_i})^{-1} (D_{K_i}+\alpha)^{-1}\delta_x (y) \notag \\
&= \frac{1}{\text{deg}(x)+\alpha}(I - (D_{K_i}+\alpha)^{-1}A_{K_i})^{-1}\delta_x (y)\notag\\
&= \frac{1}{\text{deg}(x)+\alpha} \sum_{n=0}^{\infty} \left[(D_{K_i}+\alpha)^{-1}A_{K_i}\right] ^n \delta_x (y). \label{finite approx of resolvent}
\end{align}
Furthermore, a direct calculation implies
$$ \lim_{\alpha \to 0+} \lim_{i\to \infty}\left[(D_{K_i}+\alpha)^{-1}A_{K_i}\right] ^n \delta_x (y)  = P^{(n)}(x,y).$$
Hence, in order to obtain the desired formula we need to pass to the limits under the sum in equation \eqref{finite approx of resolvent}. For this it suffices to show that  convergence in $i$ and afterwards convergence in $\alpha$ is monotone. We show by induction over $n$ that for $y \in K_i$ the inequality
$$\left[(D_{K_i}+\alpha)^{-1}A_{K_i}\right] ^n \delta_x (y) \leq \left[(D_{K_{i+1}}+\alpha)^{-1}A_{K_{i+1}}\right] ^n \delta_x (y) $$
holds. The case $n = 0$ is clear. Now assume we have shown the statement for all indices up to $n - 1$. Using the non-negativity of the quantity $\left[(D_{K_{i+1}}+\alpha)^{-1}A_{K_{i+1}}\right] ^{n-1} \delta_x (z)$ for $z \in K_{i+1}$, we obtain 
\begin{align*}
[(D_{K_i}  +\alpha &)^{-1}A_{K_i}] ^n \delta_x (y) =  \\ &= \frac{1}{\text{deg}(y)+\alpha}\sum_{z \in K_i} b(y,z)\left[(D_{K_i}+\alpha)^{-1}A_{K_i}\right] ^{n-1} \delta_x (z) \\
&\leq  \frac{1}{\text{deg}(y)+\alpha}\sum_{z \in K_i} b(y,z)\left[(D_{K_{i+1}}+\alpha)^{-1}A_{K_{i+1}}\right] ^{n-1} \delta_x (z) \\
&\leq \frac{1}{\text{deg}(y)+\alpha}\sum_{z \in K_{i+1}} b(y,z)\left[(D_{K_{i+1}}+\alpha)^{-1}A_{K_{i+1}}\right] ^{n-1} \delta_x (z)\\
&= \left[(D_{K_{i+1}}+\alpha)^{-1}A_{K_{i+1}}\right] ^n \delta_x (y).
\end{align*}
The above implies monotone convergence of the summands in $i$. A similar computation shows monotone convergence in $\alpha$. This finishes the proof.
\end{proof}

\begin{remark}
\begin{itemize}
\item The previous theorem is a version of in \cite[Theorem~4.34]{Che}. However, our proof uses a different approach. 
\item  If $(b,0)$ is a graph and $\LD$ the operator associated with $\QD$ on $\ell^2(V,1)$ then $\LD$ is a restriction of $\ow{L} = D - A$. In this sense $\LD$ equals the difference of certain $\ell^2$-restrictions of $D$ and $A$. As $\LD$, $D$ and $A$ can be unbounded on $\ell^2(V,1)$ the relation of their domains is not so clear. This is the reason why we had to use finite dimensional approximations in the previous proof. 
\item The proof of the previous theorem provides a rigorous version of the computation suggested in the discussion preceding \cite[Theorem 4.7]{JP}. Therefore, it might be considered as an $\ell^2$-analogue to \cite[Theorem 4.7]{JP}.
\end{itemize}
\end{remark}

\begin{coro} \label{equ recurrent graph recurrent df}
Let $(b,0)$ be connected and let $m$ be an arbitrary measure of full support. Let $\QD$ be the regular Dirichlet form associated with $(b,0)$ on $\ell^2(V,m)$. The random walk associated with $(b,0)$ is recurrent if and only if $\QD$ is recurrent.
\end{coro}
\begin{proof}
This is an immediate consequence of the previous two theorems and the fact that recurrence of $\QD$ does not depend on the underlying measure, see Corollary~\ref{Independence of underlying measure}.
\end{proof}

\begin{remark}
 The previous corollary seems to be widely believed among experts. However, we could not find a reference which gives the result in its full generality. A weak form of it can be found in \cite{Kal}.
\end{remark}

\begin{remark}
Let us briefly discuss the stochastic interpretation  of the formula in Theorem \ref{discrete time v.s. continuous time} by computing the involved quantities from a probabilistic view. We will omit technical details (such as construction of the related processes and measurability issues) and do computations on a formal level.

At first suppose we are given a Markov process $(X_t)_{t \geq 0}$ in continuous time with values in $V \cup \{\infty\}$ satisfying 
$$\mathbb{P}_x(X_t = y) = e^{-t\LD}\delta_y (x).$$
Here $\mathbb{P}_x$ denotes the probability under the condition that $X_0 = x$ (for a detailed discussion of the relation  of Markov processes and regular Dirichlet forms see \cite{FOT}). Let $\lambda$ be the Lebesgue measure on $\RR$. We then obtain 
\begin{align*}
\mathbb{E}_x [\lambda \{t > 0|\, X_t = y\}] &= \int_\Omega \int _0 ^\infty 1_{\{t > 0 | \, X_t(\omega) = y\}}(s) ds d\mathbb{P}_x(\omega)\\
&= \int_0 ^\infty \int _\Omega  1_{\{t > 0 | \, X_t(\omega) = y\}}(s)  d\mathbb{P}_x(\omega)ds\\
&= \int_0^\infty \mathbb{P}_x(X_s = y) ds\\
&= \int_0^\infty  e^{-s\LD}\delta_y (x) ds.
\end{align*}
Therefore, the integral $\int_0^\infty  e^{-t\LD}\delta_y (x) dt$ is equal to the expected time that $(X_t)_{t\geq 0}$ spends in $y$ provided that it started at $x$.

Now let us assume $(X_n)_{n \geq 0}$ is a random walk associated with $(b,0)$. Let $\sharp$ denote the counting measure on $\NN$. Then,
\begin{align*}
\mathbb{E}_x[\, \sharp \{n \geq 0 | \, X_n = y \}] &= \int_\Omega \sum_{n = 0} ^\infty 1_{\{\omega \in \Omega | X_n(\omega) = y \}} d\mathbb{P}_x (\omega)\\
&= \sum_{n = 0} ^\infty \mathbb{P}_x(X_n = y)\\
&=  \sum_{n = 0} ^\infty P^{(n)}(x,y).
\end{align*}
This shows that $\sum_{n = 0} ^\infty P^{(n)}(x,y)$ coincides with the expected number of visits of $(X_n)_{n\geq 0}$ to $y$ whenever it started at $x$.

\end{remark}

\section{Characterizations of Recurrence and Transience} \label{recurrence and transience}

\subsection{Classical characterizations of recurrence}

In this section we show continuous time analogues to characterizations of recurrence and transience which are known in the discrete time setting. In contrast to the discrete time theory we use the Dirichlet form methods developed above to deduce them and point out where the 'classical' results may be found. 

First we show that $\QD$ is transient whenever the graph $(b,0)$ supports a monopole of finite energy (Theorem \ref{monopol}). Afterwards, we characterize recurrence in terms of properties of the Yamasaki space $\mathbf{D}$ and the capacity of points (Theorem \ref{potential}). As a last classical characterization of recurrence, we show that it is equivalent to each superharmonic function of finite energy being constant (Theorem \ref{superharmonic}). The proofs of  Theorem \ref{monopol} and Theorem \ref{superharmonic} seem to be new while the one of Theorem \ref{potential} was suggested by Daniel Lenz.

As seen in Corollary~\ref{Independence of underlying measure}, recurrence of $\QD$ is independent of the underlying measure. Therefore, we will not indicate the $\ell^2$-space on which $\QD$ is considered in the statements of the theorems in this section.  
\begin{definition}[Monopole] A function $u \in \ow{D}$ is called a \emph{monopole of finite energy} if there exists some $x \in V$ such that $u$ satisfies
$$ \ow{L}u = \delta_x.$$
\end{definition}

The first theorem we  prove deals with the existence of such monopoles. Its discrete-time analogue is due to Lyons  \cite{Lyo} (see \cite[Theorem~3.33]{Soa} for related material).

\begin{theorem} \label{monopol}
Let $(b,0)$ be connected. The Dirichlet form $\QD$ is transient if and only if there exists a monopole of finite energy.
\end{theorem}

\begin{proof}
Assume $u \in \ow{D}$ satisfies $\ow{L}u = \delta_w$. We show the transience of $\QD$ by constructing a reference function $g$ as in Definition \ref{Definition transient space}. Let $v \in  C_c(V)$ be given. By Lemma \ref{green} and the Cauchy-Schwarz inequality we obtain
$$|v(w)|m(w) = \as{v,\delta_w} = \as{v,\ow{L}u} = \ow{Q}(v,u) \leq \QD(v)^{1/2}\ow{Q}(u)^{1/2}.$$
Since $C_c(V)$ is dense in $D(\QD)$ with respect to the form norm $\aV{\cdot}_Q$, this inequality extends to all $v \in D(\QD)$. Furthermore, Lemma \ref{difference functional} shows that for every $x\in V$ there exists a constant $K_x > 0$ with 
$$|v(w)-v(x)| \leq K_x \QD(v)^{1/2}$$
for every $v \in D(\QD)$. Combining these two inequalities we infer the existence of $C_x > 0$ such that for every $v \in D(\QD)$ the inequality
$$|v(x)| \leq C_x \QD(v)^{1/2}$$
holds. Now we define a reference function $g$ by setting
$$g(x) = \frac{a_x}{C_x m(x)},$$
where the $a_x$ are chosen strictly positive such that $g \in \ell^1(V,m) \cap \ell^{\infty}(V)$ and $\sum a_x = 1$. This finishes the first part of the proof.

On the contrary, let us assume $\QD$ is transient. By Definition \ref{Definition transient space} there exists a strictly positive $g \in \ell^{\infty}(V) \cap \ell^1(V,m)$ such that 
$$ \as{|u|,g} \leq \QD(u)^{1/2}, \text{ for every } u\in D(\QD).$$ 
By the definition of $D(\QD)_e$ and Lemma \ref{extended form}, the above inequality extends to all $u \in D(\QD)_e$. For a fixed $w \in V$, this implies the continuity of the linear functional 
$$F_w: D(\QD)_e \to \RR, \, u \mapsto u(w),$$  
with respect to the inner product $\QD$. Theorem \ref{char transience} yields that  $(D(\QD)_e, \QD)$ is a Hilbert space. Thus, by Riesz representation theorem there exists a function $v \in D(\QD)_e$ such that 
$$ u(w) = F_w(u) =  \QD(u,v)$$
for all $u \in D(\QD)_e$. By Lemma \ref{extended form} the inclusion $D(\QD)_e \subseteq \ow{D}$ holds. With the help of Lemma \ref{green} we  compute
$$(\ow{L}v)(x) = \frac{1}{m(x)}\as{\ow{L}v,\delta_x} = \frac{1}{m(x)}\ow{Q}(v,\delta_x) = \frac{1}{m(x)}\QD(v,\delta_x) = \frac{\delta_x(w)}{m(x)}.$$
This shows $\ow{L}(m(w)v) = \delta_w$ and finishes the proof. 
\end{proof}

For locally finite graphs the next classical characterization of recurrence is due to Yamasaki \cite{Yam1}.  It deals with the structure of the space $\mathbf{D}$ and shows that recurrence is equivalent to points having capacity zero, where the \emph{capacity} of $x \in V$ is defined by 
$$\text{cap}(x) =\inf \{\QD(v)\mid \, v \in C_c(V), \,  v(x)= 1\}.$$
\begin{theorem}\label{potential} Let $(b,0)$ be connected. The following assertions are equivalent:
\begin{itemize}
\setlength{\itemsep}{0pt}
\item[(i)] $\QD$ is recurrent.
\item[(ii)] $C_c(V)$ is dense in $\mathbf{D}$, i.e., $\mathbf{D} = \mathbf{D}_0$.
\item[(iii)] The constant function $1$ can be approximated in $\mathbf{D}$ by functions of $C_c(V)$. In this case, the approximating functions $e_n$ can be chosen to satisfy $0 \leq e_n \leq 1$. 
\item[(iv)] $\text{\emph{cap}}(o) = \inf \{\QD(v)\mid \, v \in C_c(V), \,  v(o)= 1\} = 0$.
\end{itemize}
\end{theorem}

\begin{proof}
'(i) $\Rightarrow$ (iii)':  Theorem \ref{char recurrence} yields the existence of a sequence $(f_n) \subseteq D(\QD)$, such that $f_n \to 1$ with respect to $\aV{\cdot}_o$. Since $C_c(V)$ is dense in $D(\QD)$ with respect to $\aV{\cdot}_Q$ there exist $\tilde{e}_n \in C_c(V)$ such that 
$$\aV{\tilde{e}_n - f_n}_Q \to 0 \text{ as } n \to \infty.$$ 
Let $e_n = (0 \vee \tilde{e}_n)\wedge 1$. By the Markov property of $\QD$ we obtain
$$\QD(e_n)^{1/2} \leq \QD(\tilde{e}_n)^{1/2} \leq \QD(\tilde{e}_n- f_n)^{1/2} +  \QD(f_n)^{1/2}.$$
Because $c \equiv 0$, the right side of the above inequality needs to converges to zero as $n\to \infty$. It is straightforward that $e_n \to 1$ pointwise. This implies 
$$\aV{1-e_n}_o \to 0 \text{ as } n\to \infty,$$
and shows (iii).

'(iii) $\Rightarrow$(ii)': This proof will be done in two steps.

\emph{Step 1}:  Let $u \in \mathbf{D}$, such that $0 \leq u \leq 1$ and let $(e_n) \subseteq C_c(V)$ be a sequence approximating $1$ in $\mathbf{D}$, such that $0 \leq e_n \leq 1$. Then, by Corollary \ref{small convergence} the sequence $u \wedge e_n$ converges to $u$ with respect to $\aV{\cdot}_o$. Furthermore, $u \wedge e_n \in C_c(V)$ showing that
$$u \in \overline{C_c(V)}^{\aV{\cdot}_o}.$$

\emph{Step 2}: Let $u \in \mathbf{D}$ such that $u \geq 0$. Then, Corollary \ref{bounded convergence} yields the convergence of $u\wedge N$ to $u$ with respect to $\aV{\cdot}_o$ as $N \to \infty$. Step 1 allows us to approximate $ u \wedge N$ by functions of $C_c(V)$. For general $u \in \mathbf{D}$ we can split $u$ in its positive and negative part which both belong to $\mathbf{D}$.  This shows $(ii)$.

'(ii) $\Rightarrow$ (i)': Proposition \ref{char extended space} shows that $D(\QD)_e$ is the closure of $D(\QD)$ in $\mathbf{D}$. Since $C_c(V) \subseteq D(\QD)$, condition (ii) implies the equality $D(\QD)_e = \mathbf{D}$.  From this and $c \equiv 0$ we infer $1 \in D(\QD)_e $ and  $\QD(1) = 0$. Now Theorem \ref{char recurrence} yields (i).

'(iii) $\Rightarrow$ (iv)': This is obvious noting that the sequence in (iii) can be chosen to satisfy $e_n(o) = 1$.

'(iv) $\Rightarrow$ (i)': Assume $\QD$ is transient. By Definition \ref{Definition transient space} there exists a constant $C > 0$ such that for any $v \in D(\QD)$ the inequality  $\QD(v)^{1/2} \geq C |v(o)|$ holds. In particular,
$$\inf \{\QD(v) \mid \, v \in C_c(V), \,  v(o)= 1\} \geq C^2 > 0. $$
This finishes the proof.
\end{proof}

\begin{remark}
  For further references on the history of the previous theorem and related results see \cite[Chapter~3.7]{Soa}.
\end{remark}

The last classical which we prove deals with \emph{superharmonic functions of finite energy}, i.e., functions $u \in \ow{D}$ satisfying 
$$\ow{L}u \geq 0.$$
It is an  analogous to Theorem 3.34 of \cite{Soa}.

\begin{theorem} \label{superharmonic}
Let $(b,0)$ be connected. The Dirichlet form $\QD$ is recurrent if and only if any superharmonic function of finite energy is constant.
\end{theorem}

\begin{proof}
Assume $\QD$ is recurrent and let $u \in \ow{D}$ with $\ow{L}u \geq 0$ be given. As a first step we show that $u$ is harmonic, i.e., $  \ow{L}u = 0$.  Assume that there exists a $w \in V$ such that $\ow{L}u(w) > 0$. By Lemma \ref{green} we obtain for all $v \in C_c(V)$
$$|v(w)|(\ow{L}u)(w)m(w) \leq \as{|v|,\ow{L}u} = \ow{Q}(|v|,u) \leq \QD(v)^{1/2} \ow{Q}(u)^{1/2}.$$
Following the first part of the proof of Theorem \ref{monopol}, such an inequality implies transience of $\QD$. Hence, we conclude $\ow{L}u = 0$. 

Next, we show $|u(x)-u(y)| = 0$ for all $x,y \in V$. Since $\QD$ is recurrent part $(ii)$ of Theorem \ref{potential} implies the existence of a sequence $(u_n) \subseteq C_c(V)$ with $\aV{u-u_n}_o \to 0$. Furthermore, by Lemma \ref{difference functional} for each $x,y \in V$ there exists a constant $K_{x,y} > 0,$ such that 
$$|u(x)-u(y)| \leq K_{x,y}\ow{Q}(u)^{1/2}.$$ 
Combining these obervations and Lemma \ref{green} we obtain
\begin{align*}
|u(x)-u(y)| &\leq  K_{x,y}\ow{Q}(u)^{1/2} = K_{x,y}\lim_{n \to \infty}\ow{Q}(u,u_n)^{1/2} \\&= K_{x,y}\lim_{n \to \infty}\as{\ow{L}u,u_n}^{1/2}=0.
\end{align*}
This proves one implication.

On the contrary assume $\QD$ is transient. By Theorem \ref{monopol} there exists a monopole of finite energy. This monopole clearly is superharmonic and nonconstant.
\end{proof}

\begin{remark}
In the literature the normalized Laplacian, i.e., the operator $\ow{L}_{b,0,\deg}$ is used to state analogous theorems to Theorem \ref{monopol} and Theorem \ref{superharmonic} for the discrete time case. 
\end{remark}

\subsection{New criteria for recurrence}

In this section we provide two more criteria for recurrence which seem to be new. The first one asks whether certain integrals vanish (Theorem \ref{integral recurrence}), while the second one deals with the validity of Green's formula for a different situation than in Lemma \ref{green} (Theorem \ref{boundary recurrence}). Both criteria were motivated by recent works. The first one is an analogue to a result of \cite{GM}, while the second one is a version of Theorem 4.6 in \cite{JP} for not necessarily locally finite graphs.  

\begin{theorem} \label{integral recurrence}
Let $(b,0)$ be connected. The form $\QD$ is recurrent if and only if 
$$\sum_{x \in V} \ow{L}u(x)m(x) = 0$$
for all $u \in \ow{D}$ with $\ow{L}u \in \ell^1(V,m)$.
\end{theorem}

\begin{proof}
Let $\QD$ be recurrent and $u  \in \ow{D}$, such that $\ow{L}u \in \ell^1(V,m)$. By Theorem \ref{potential} there exists a sequence $e_n$ in $C_c(V)$ satisfying $\aV{e_n - 1}_o \to 0$ and $0 \leq e_n \leq 1$. We infer by Lebesgue's theorem and Lemma \ref{green} 
\begin{align*}
 \sum_{x \in V} \ow{L}u(x)m(x) & = \lim_{n \to \infty}\sum_{x \in V} e_n(x) \ow{L}u(x)m(x) \\
 &= \lim_{n \to \infty} \ow{Q}(e_n, u) = 0.
\end{align*}
On the contrary assume $\QD$ is transient. Then, Theorem \ref{monopol} yields the existence of a function $v \in \ow{D}$ satisfying $\ow{L}v = \delta _w \in \ell^1(V,m)$ for some $w \in V$. Obviously 
$$\sum_{x \in V} \ow{L}v(x)m(x) = m(w) \neq 0.$$
This finishes the proof.
\end{proof}

For proving Green's formula in Lemma \ref{green} we needed that one of the functions had compact support. As a last characterization for recurrence we show that recurrence is equivalent to the validity of Green's formula for a different class of functions. We introduce the \emph{boundary term} 
$$R:D_\infty \times D^1 \to \RR$$
 by
$$R(u,v) = \ow{Q}(u,v) - \as{u,\ow{L}v}.$$
Here we used the notation $D_\infty = \ow{D} \cap \ell^{\infty}(V)$ and $D^1 = \{v \in \ow{D}\,| \, \ow{L}v \in \ell^1(V,m)\}$.

\begin{theorem} \label{boundary recurrence} Let $(b,0)$ be connected. The Dirichlet form $\QD$ is recurrent if and only if $R \equiv 0$.  
\end{theorem}

\begin{proof}
Assume $\QD$ is recurrent. Let $u \in D_\infty $ and $v \in D^1$ be given. Then, Theorem \ref{potential} yields the existence of a sequence $(u_n) \subseteq C_c(V)$ converging to $u$ with respect to $\aV{\cdot}_o$. Without loss of generality, this sequence can be chosen to be uniformly bounded by $\aV{u}_\infty$. Then, Lebesgue's theorem and Lemma \ref{green} yield
$$\ow{Q}(u,v) = \lim_{n\to \infty}\ow{Q}(u_n,v) = \lim_{n\to \infty}\as{u_n, \ow{L}v} = \as{u,\ow{L}v}.$$
This implies $R \equiv 0$.

On the contrary assume $\QD$ is transient. By Theorem \ref{monopol} there exists a function $v \in D^1$, such that 
$$ \sum_{x \in V} \ow{L}v(x)m(x) \neq 0.$$
Since $1 \in D_\infty$ and $\ow{Q}(1,v) = 0$, this implies $R(1,v) \neq 0$ which finishes the proof.

\end{proof}

\begin{remark} 
The previous theorem was motivated by \cite[Theorem~4.6]{JP}. This theorem deals with a boundary term pairing a space spanned by certain monopoles and dipoles and the functions of finite energy. However, the boundary representation that is used to deduce the result there seems to hold true only for locally finite graphs. We will explain some details of their computation below. 
\end{remark}

Let us now consider the case when $(b,0)$ is locally finite. We can then compute the boundary term $R$ by a limiting procedure. First we fix some notation. For a subgraph $W \subseteq V$ let 
$$\text{bd} \, W = \{x \in W| \text{ there exists }y \in V \setminus W \text{ such that } x \sim y\}$$
be the set of all vertices in $W$  which are connected with the complement of $W$. Furthermore, let 
$$\text{int}\, W = W\setminus \text{bd}W.$$
Note that $x \in \text{int}W$ and $y \sim x$ implies $y \in W$. For $u \in \ow{D}$ and $x \in \text{bd}W$ we let the outward normal derivative with respect to $W$ be defined by
$$(\partial_W u)(x) = \sum_{y \in W}b(x,y)(u(x)-u(y)). $$
With it we can compute the boundary term $R$ as in \cite{JP}. 

\begin{prop}
Let $(b,0)$ be locally finite. For $u \in D_\infty$ and $v \in D^1$ the boundary term $R$ is given by 
 $$R(u,v) = \lim_{n \to \infty} \sum_{x \in \text{bd}V_n}u(x)(\partial_{V_n} v)(x),$$
where $(V_n)$ is an increasing sequence of finite subsets of $V$ with $\cup _n V_n = V.$
\end{prop}

\begin{proof}
Let $V_n$ be as above. Then, a simple calculation shows
\begin{align*}
\frac{1}{2}\sum_{x,y \in V_n}b(x,y)(u(x)&-u(y))(v(x)-v(y)) \\&= \sum_{x \in \text{int}V_n} u(x)(\ow{L}v)(x)m(x) + \sum_{x \in \text{bd}V_n}u(x)(\partial_{V_n} v)(x).
\end{align*}
Because $(b,0)$ is locally finite, we obtain $\cup_n \text{int}V_n = V$. Furthermore, our assumptions yield that the sum on the left and 
$$\sum_{x \in V} u(x)(\ow{L}v)(x)m(x)$$
are absolutely convergent. Taking the limit $n \to \infty$ implies the desired statement, noting that absolute convergence yields independence of the choice of the $V_n$.
\end{proof}
 
\begin{remark}
The local finiteness is crucial for the above computations, which are taken from the proof of Theorem 4.6 in \cite{JP}. Otherwise one cannot control $\text{\emph{int} }W$ for finite sets $W$. In the non local finite case it might even happen that $\text{\emph{int} }W = \emptyset$ for all finite $W \subseteq V$.
\end{remark}



\section{Further global properties} \label{further global properties}

In this chapter we  discuss two other concepts - namely stochastic completeness and the validity of $\QD = \QN$. It turns out that characterizations of these two properties are similar to the ones obtained for recurrence and transience. We first introduce the notion of stochastic completeness (Definition \ref{SC}) and then prove a characterization which  is analogous to Theorem \ref{integral recurrence} (see Theorem \ref{integral sc}) and is also motivated by results of \cite{GM}. Afterwards, we present a criterion for stochastic completeness in terms of the unique solvability of the equation $(\ow{L} + \alpha)u = 0$ on $\ell^\infty$ (Theorem \ref{unique solution sc}). This criterion is taken from \cite{KL}. We then characterize when the Neumann form $\QN$ and the regular Dirichlet form $\QD$ coincide. We show that this is related to unique solvability of $(\ow{L} + \alpha)u = 0$ on $\ow{D} \cap \ell^2(V,m)$ and the validity of Green's formula for $\ell^2$-functions (Theorem \ref{unique solvability QN=QD}). The connection between $\QD = \QN $ and the validity of $\ow{Q}(u,v) = \as{\ow{L}u,v}$ for certain  $\ell^2$-functions seems to be new.

\subsection{Stochastic completeness}

Using the extension of a Markovian resolvent to $\ell^\infty(V)$ (see Appendix~\ref{appendix:dirichlet forms}) we  introduce the concept of stochastic completeness.  

\begin{definition}\label{SC}
Let $Q$ be a Dirichlet form associated with $(b,c)$. $Q$ is called \emph{stochastically complete} if $$(L+1)^{-1}1 = 1.$$
Otherwise $Q$ is called \emph{stochastically incomplete}. 
\end{definition}

\begin{remark}
By general principles (the correspondence of $(L+\alpha)^{-1}$ and $e^{-tL}$) the  definition of stochastic completeness is equivalent to the validity of 
$$e^{-tL}1 = 1$$
for all $t > 0$. This equation is important whenever one investigates a Markov process $(X_t)_{t \geq 0}$ on $V\cup \{\infty\}$ which satisfies 
$$\mathbb{P}(X_t = y | X_0 = x) = e^{-tL}\delta_y(x).$$
In view of this equation stochastic completeness is equivalent to 
$$\mathbb{P}(X_t \in V | X_0 = x) = 1 \text{ for all } t>0.$$
In other words, stochastic completeness describes the property that $X_t$ does not leave $V$ in finite time.
\end{remark}

The next result is similar to Theorem \ref{integral recurrence}. It seems to be new in this context.

\begin{theorem} \label{integral sc}
Let $(b,c)$ be connected and let $m$ be a measure of full support. The associated regular Dirichlet form $\QD$ on $\ell^2(V,m)$ is stochastically complete if and only if the equality
$$\sum_{x\in V} (\ow{L}u)(x)m(x) = 0$$
holds for all $u \in D(\QD) \cap \ell^1(V,m)$ with $\ow{L}u \in \ell^1(V,m) \cap \ell^2(V,m)$.
\end{theorem}

\begin{proof} Let $(e_n)$ be a sequence in $C_c(V)$  which satisfies $0 \leq e_n \leq e_{n+1} \leq 1$ and $e_n \to 1$ pointwise. Let $u_n = (\LD + 1)^{-1}e_n \in D(\LD)$. The way we extended the resolvent to $\ell^\infty$ yields pointwise convergence of $(u_n)$  towards $(\LD + 1)^{-1}1$. Furthermore, because $(\LD + 1)^{-1}$ is a positivity preserving, Markovian resolvent, we infer $0 \leq u_n \leq 1$.

Now assume $\QD$ is stochastically complete and let $u \in D(\QD) \cap \ell^1(V,m)$, such that $\ow{L}u \in \ell^1(V,m) \cap \ell^2(V,m)$ be given. Proposition \ref{QD} shows that $u$ belongs to $D(\LD)$. Furthermore, stochastic completeness yields pointwise convergence of $(u_n)$ towards $1$. Using the self-adjointness of $\LD$ and Lebesgue's theorem we may compute 
\begin{align*}
\sum_{x\in V} (\ow{L}u)(x)m(x) &= \sum_{x\in V} (\LD u)(x)m(x)\\
&= \lim_{n \to \infty} \sum_{x\in V}u_n(x) (\LD u)(x)m(x)\\
&= \lim_{n \to \infty} \sum_{x\in V}(\LD u_n)(x) u(x)m(x)\\
&= \lim_{n \to \infty} \sum_{x\in V}(e_n(x)- u_n(x)) u(x)m(x)\\
&= \sum_{x\in V}(1-(\LD + 1)^{-1}1(x)) u(x)m(x)\\
&= 0.
\end{align*}
This shows one implication.

On the contrary, if the sum is always vanishing, put $u = (\LD + 1)^{-1}v$, where $v \in \ell^1(V,m)\cap \ell^2(V,m)$ is chosen strictly positive. This implies that $u$ belongs to $ D(\QD) \cap \ell^1(V,m)$ (see appendix, extension of the resolvent to $\ell^1$) and $\ow{L}u \in \ell^1(V,m) \cap \ell^2(V,m)$. Then, our assumptions, Lebesgue's theorem and the self-adjointness of the resolvent yield
\begin{align*}
0 &= \sum_{x \in V} (\ow{L}u)(x)m(x)\\
 &= \lim_{n\to \infty} \sum_{x \in V} e_n(x)(\LD u)(x)m(x)\\
 &= \lim_{n\to \infty} \sum_{x \in V} e_n(x)(v(x) - (\LD +  1)^{-1}v(x))m(x)\\
 &= \lim_{n\to \infty} \sum_{x \in V} (e_n(x) - (\LD +  1)^{-1}e_n(x))v(x)m(x)\\
 &=  \sum_{x \in V} (1 - (\LD +  1)^{-1}1(x))v(x)m(x).
\end{align*}
Since $v$ was chosen strictly positive and $(\LD +  1)^{-1}1(x) \leq 1$ (see appendix) this shows $(\LD +  1)^{-1}1 = 1$.
\end{proof}

\begin{remark}
As for recurrence one can show that on a connected graph  a nonvanishing potential $c$ implies stochastic incompleteness (see e.g. \cite{KL}). Therefore, we will assume $c \equiv 0$ in the following sections.
\end{remark}

The next theorem is a characterization of stochastic completeness in terms of the unique solvability of $(\ow{L} + \alpha)u = 0$ on $\ell^\infty(V)$. We will need it for the discussion in the next chapter.

\begin{theorem} \label{unique solution sc}
Let $\QD$ be the regular Dirichlet form associated with $(b,0)$ on $\ell^2(X,m)$. The following assertions are equivalent:
\begin{itemize}
 	\item[(i)]$\QD$ is stochastically complete.
 	\item[(ii)] For any $\alpha > 0$ the equation $(\ow{L} + \alpha)u = 0$ is uniquely solvable on $\ell^\infty (V)$.
\end{itemize}
\end{theorem} 

\begin{proof}
This is an immediate consequence of Theorem 1 in \cite{KL}.
\end{proof}

\subsection{Regularity of the Neumann form}

From the definition of $\QD$ and $\QN$ it is not clear whether these two forms coincide or not. The theorem below provides a characterization of this in terms of unique solvability of $(\ow{L} + \alpha)u = 0$ on $\ow{D}\cap \ell^2(V,m)$ and the validity of Green's formula for $\ell^2$-functions.

\begin{theorem}
Let $\QD$ be the regular Dirichlet form associated with $(b,c)$ on $\ell^2(X,m)$ and let $\QN$ be Neumann form associated with $(b,c)$ on $\ell^2(X,m)$. The following assertions are equivalent:
\begin{itemize} \label{unique solvability QN=QD}
 	\item[(i)]$\QD = \QN.$ 
 	\item[(ii)] For any $\alpha > 0$ the equation $(\ow{L} + \alpha)u = 0$ is uniquely solvable in $\ow{D}\cap \ell^2(V,m)$.
 	\item[(iii)] For all $u \in \ow{D}\cap \ell^2(V,m)$ and $v \in \ow{D}\cap \ell^2(V,m)$ with $\ow{L}v \in \ell^2(V,m)$ the equation
$$\ow{Q}(u,v) = \as{u,\ow{L}v}$$
holds. 
 \end{itemize}
\end{theorem}

\begin{proof}
'(ii) $\Rightarrow$ (i)': By the general theory (see Appendix~\ref{appendix:dirichlet forms}) it suffices to show that the resolvents $(\LD+\alpha)^{-1}$ and $(\LN+\alpha)^{-1}$ coincide. Let $u\in \ell^2(V,m)$ be arbitrary. Set $$v = (\LN+\alpha)^{-1}u - (\LD+\alpha)^{-1}u.$$ Since both operators $\LD$ and $\LN$ are restrictions of $\ow{L}$ to their corresponding domains, we infer 
$$(\ow{L}+\alpha)v = 0.$$
Hence, assertion (ii) implies $v = 0$ and we conclude  $(\LD+\alpha)^{-1} =(\LN+\alpha)^{-1}$.

'(i) $\Rightarrow$ (ii)': By Proposition \ref{QD} the domain of $\LD$ is given by
$$D(\LD) = \{u \in D(\QD) \mid \ow{L}u \in \ell^2(V,m)\}.$$
Since we assumed $\QD = \QN$ this implies
$$D(\LD) =  \{u \in \ow{D}\cap \ell^2(V,m)\mid  \ow{L}u \in \ell^2(V,m)\}.$$
Let $\alpha > 0$ and $u \in \ow{D}\cap \ell^2(V,m)$  with $\ow{L}u = -\alpha u$ be given. By the above characterization of $D(\LD)$ this implies $u \in D(\LD)$.  Since the spectrum of $\LD$ is contained in $[0,\infty)$, we infer $u = 0.$

'(i) $\Rightarrow$ (iii)': Assume $\QD = \QN$. Proposition \ref{QD} implies
\begin{align*}
 D(\LN) = D(\LD) &= \{v \in D(\QD)\mid \ow{L}v \in \ell^2(V,m)\} \\&= \{v \in \ow{D}\cap \ell^2(V,m)\,| \, \ow{L}v \in \ell^2(V,m)\}.
\end{align*}

This shows (iii).

'(iii) $\Rightarrow$ (i)': Assume $\ow{Q}(u,v) = \as{u,\ow{L}v}$ for all $u \in \ow{D}\cap \ell^2(V,m)$ and $v \in \ow{D}\cap \ell^2(V,m)$ with $\ow{L}v \in \ell^2(V,m)$. By the correspondence of $\LN$  and $\QN$ the domain of $\LN$ satisfies
$$D(\LN) \supseteq \{v \in \ow{D}\cap \ell^2(V,m)\,| \, \ow{L}v \in \ell^2(V,m)\}.$$
Hence, Proposition \ref{QD} shows $\LD \subseteq \LN$. Taking adjoints yields the statement.
\end{proof}

\begin{remark}
The equivalence (i) and (ii) was already shown in \cite[Corollary 3.3]{HKLW}. However their proof is different from the one given above. The equivalence of (i) and (iii) seems to be new. 
\end{remark}

Part (iii) of the theorem above can be considered as a boundary term characterization of $\QD = \QN$ which is an analogous to Theorem~ \ref{boundary recurrence}. To see this we introduce the boundary term 
$$\hat{R}: \ow{D}\cap \ell^2(V,m) \times \{u \in\ow{D}\cap \ell^2(V,m) \, | \, \ow{L}u \in \ell^2(V,m) \} \to \RR$$
acting by
$$\hat{R} (u,v) = \ow{Q}(u,v) - \as{u,\ow{L}u}.$$
\begin{coro}
The equality $\QD = \QN$ holds if and only if $\hat{R} \equiv 0.$
\end{coro}


\section{Consequences of recurrence} \label{consequences of recurrence}

In this chapter we discuss the relationship between all the global properties which were introduced above. We prove that recurrence of $\QD$ always implies stochastic completeness and $\QD = \QN$ (see Theorem \ref{implications of recurrence}) and that all concepts coincide in the case when $m$ is finite (Theorem \ref{finite measure}). Using these results we show that recurrence of $\QD$ is related to the unique solvability of the eigenvalue equation $(\ow{L} + \alpha)u =0$ on the space $\ow{D}$. (Theorem \ref{last}). Except Theorem \ref{implications of recurrence}, which is valid in much more general situations,  all the results of this chapter seem to be new.

\begin{lemma}
Let $(b,0)$ be connected and assume that $\QD$ is recurrent. Let $\alpha > 0$  and let $u \in \ow{D}$ satisfy $u \leq 0$ and  $(\ow{L}+\alpha)u \geq 0$. Then $u \equiv 0$.
\end{lemma}
\begin{proof}
Let $u$ be as above. By $\ow{L}u \geq -\alpha u$ we infer that $u$ is superharmonic. Theorem \ref{superharmonic} implies that $u$ is constant. We obtain
$$0 \leq (\ow{L}+\alpha)u = \alpha u \leq 0. $$
This shows $u  \equiv 0$ and finishes the proof.
\end{proof}

From this lemma we can deduce the following uniqueness statement for solutions to the equation $(\ow{L}+\alpha)u = 0$ in $\ow{D}$.

\begin{lemma}
	Let $(b,0)$ be connected and assume that $\QD$ is recurrent. Ler $\alpha > 0$ and let $u \in \ow{D}$ with $(\ow{L} + \alpha)u = 0$. Then $u \equiv 0$.
\end{lemma}
\begin{proof}
Let $u \in \ow{D}$  with  $(\ow{L}+\alpha)u = 0$ be given. Let $u_+ = u \wedge 0$ and $u_- = (-u)\wedge0$ denote the positive/negative part of $u$. Since $u_+, u_- \in \ow{D}$  it suffices to show $(\ow{L}+\alpha)u_+ \leq 0$ and $(\ow{L}+\alpha)u_- \leq 0$ to obtain the statement by the previous lemma. The assumption $(\ow{L}+\alpha)u = 0$ implies 
$$(\ow{L}+\alpha)u_+ = (\ow{L}+\alpha)u_-, $$
which is equivalent to
$$(\ow{L}+\alpha)u_+(x) = \frac{\text{deg}(x)}{m(x)}u_-(x) - \frac{1}{m(x)}\sum_{y\in V} b(x,y)u_-(y) + \alpha u_-(x).$$
For $x\in V$ with $u(x) \geq 0$ we obtain 
$$(\ow{L}+\alpha)u_+(x) =  - \frac{1}{m(x)}\sum_{y\in V} b(x,y)u_-(y) \leq 0.$$
For $x\in V$ with  $u(x) < 0 $ definition of $\ow{L}$ yields
	$$(\ow{L}+\alpha)u_+(x) = - \frac{1}{m(x)}\sum_{y\in V} b(x,y)u_+(y) \leq 0.$$
This finishes the proof.
\end{proof}

We are now able to prove that recurrence implies stochastic completeness and $\QD = \QN$.

\begin{theorem} \label{implications of recurrence}
Let $(b,0)$ be connected and $\QD$ be recurrent. Then, $\QD$ is stochastically complete and $\QD = \QN$.
\end{theorem}
\begin{proof}
Stochastic completeness: This is an immediate consequence of Theorem \ref{integral recurrence} and Theorem \ref{integral sc}.

$\QD = \QN$: This is an immediate consequence of Lemma 7.2 and Theorem \ref{unique solvability QN=QD}.
\end{proof}

\begin{theorem} \label{finite measure}
Let $(b,0)$ be connected and $m(V) < \infty$. The following assertion are equivalent:
\begin{itemize}
\setlength{\itemsep}{0pt}
\item[(i)]$\QD$ is recurrent.
\item[(ii)]  $\QD$ is stochastically complete.
\item[(iii)] $\QD = \QN$.
\end{itemize}

\end{theorem}
\begin{proof}
'(i) $\Rightarrow$ (ii)': This is one implication of Theorem \ref{implications of recurrence}.

'(iii) $\Rightarrow$ (i)': The finiteness of $m$ implies $1 \in D(\QN) = D(\QD)$. Since  $c \equiv 0,$ we conclude $\QD(1) = 0$ and recurrence of $\QD$ follows from Theorem \ref{char recurrence}.

'(ii) $\Rightarrow$ (iii)': Let us assume $\QD \neq \QN$. Then, the resolvents $(\LN+1)^{-1}$ and $(\LD+1)^{-1}$ must be different. Because $\ell^\infty(V)\cap \ell^2(V,m)$ is dense in $\ell^2(V,m)$, there exists a bounded function $u $  with
$$(\LD+1)^{-1}u \neq  (\LN+1)^{-1}u.$$
Both functions $(\LD+1)^{-1}u $ and $(\LN+1)^{-1}u$ are bounded solutions to the equation 
$$(\ow{L}+1)v = u.$$
Therefore, $(\ow{L}+1)$ is not injective on $\ell^\infty (V)$. This implies stochastic incompleteness by Theorem \ref{unique solution sc}.
\end{proof}

We have already seen that stochastic completeness and the equality $\QD = \QN$ are related to the uniqueness of solutions of the equation $(\ow{L} + \alpha)u = 0$ on certain function spaces. With the help of the last theorem we prove a similar statement for recurrence. However, as we pointed out before, recurrence does not depend on the choice of $m$. Thus, the uniqueness statement which is equivalent to recurrence must be stronger than the ones for the other concepts.  Let us write $\QD_m$ whenever we refer to $\QD$ on $\ell^2(V,m)$ and let $\ow{L}_m$ be the corresponding formal operator. Furthermore, by $\ow{\Delta}$ we denote the  operator on $\ow{D}$ that acts by 
$$
(\ow{\Delta}u)(x)  := (\ow{L}_{b,0,1}u)(x) = \sum_{y \in V}b(x,y)(u(x)-u(y)).
$$
\begin{theorem} \label{last}
Let $(b,0)$ be connected. The following assertion are equivalent:
\begin{itemize}
\item[(i)] For some measure of full support $m$ on $V$ the form $\QD_m$ is recurrent. 
\item[(ii)] For all measures of full support $m$ on $V$ the form $\QD_m$ is recurrent.
\item[(iii)] For all measures of full support $m$ and for any $\alpha >0$ the equation $(\ow{L}_m + \alpha)u = 0$ has a unique solution in $\ow{D}$.
\item[(iv)] For some finite measure of full support $m$ and for any $\alpha >0$ the equation $(\ow{L}_m + \alpha)u = 0$ has a unique solution in $\ow{D}$.
\item[(v)] For all $v:V \to (0,\infty)$ the equation $(\ow{\Delta} + v)u = 0$ has a unique solution in $\ow{D}$. 
\item[(vi)] For some $v:V \to (0,\infty)$ which belongs to $\ell^1(V,1)$ the equation $(\ow{\Delta} +  v)u = 0$ has a unique solution in $\ow{D}$.

\end{itemize}
\end{theorem}
\begin{proof}
'(i)$\Rightarrow$ (ii)': It suffices to show the statement for transience. Let $m$ be a measure, such that $\QD_m$ is transient and let $m'$ be another measure of full support. Theorem \ref{monopol} shows that there exists a function $v \in \ow{D}$ and $w \in V$, such that 
$$\ow{L}_m v = \delta_w.$$
Then, 
$$\ow{L}_{m'} v = \frac{m(w)}{m'(w)}\delta_w.$$
This implies the existence of a monopole with respect to $\ow{L}_{m'}$ and Theorem \ref{monopol} shows transience of $\QD_{m'}$.

'(ii)$\Rightarrow$ (iii)': This follows from Lemma 7.2.

'(iii)$\Rightarrow$ (iv)': This is obvious.

'(iv)$\Rightarrow$ (i)': By the uniqueness of the solution, we infer $\QD_m = \QN_m$ from Theorem \ref{unique solvability QN=QD}. Since $m$ is finite, Theorem \ref{finite measure} implies recurrence of $\QD_m$.

'(v)$\Leftrightarrow$ (iii)': This is clear since any $v:V \to (0,\infty)$ can be written in the form $v = \alpha m$ and vice versa. Then,  $(\ow{\Delta} + v)u = 0$ if and only if $(\ow{L}_m + \alpha)u = 0$.

'(v)$\Rightarrow$ (vi)': This is clear.

'(vi) $\Rightarrow$ (i)': Let $v:V \to (0,\infty),$ which belongs to $\ell^1(V,1),$ such that
$$(\ow{\Delta} +  v)u = 0 \label{equation}$$
has a unique solution in $\ow{D}$. Then, the above is obviously equivalent to 
$$(\ow{L}_v + 1)u = 0$$
being uniquely solvable in $\ow{D}$. We infer that $(\LD + 1)^{-1}$ and $(\LN+1)^{-1}$ must agree (as resolvents associated with $\QD,\QN$ on  $\ell^2(V,v)$). This shows $\LD = \LN$ which implies $\QD = \QN$. Since $v$ is a finite measure on $V$, we can conclude by Theorem \ref{finite measure} that $\QD_v$ is recurrent arriving at (i).
\end{proof}



\appendix

\section{General results}

In this appendix we provide known results, which would not fit in the main text.  The first part is devoted to general theory of Dirichlet forms, while the second one deals with some results about Dirichlet forms on graphs. The last part provides two theorems about vector valued integrals.

\subsection{Dirichlet forms and associated objects} \label{appendix:dirichlet forms}

\begin{definition} Let $D(Q)$ be a dense subspace of $\ell^2(V,m)$. A map 
$$Q: D(Q) \times D(Q) \rightarrow \RR$$
is called Dirichlet form if the following conditions are satisfied:
\begin{itemize}
\item[$(Q1)$] $Q(f,f) \geq 0$, $Q(f,g) = Q(g,f)$ and $Q(\alpha f + g, h) = \alpha Q(f,h) + Q(g,h)$ for all $f,g,h \in D(Q), \alpha \in \RR.$ \hfill (Linearity)
\item[$(Q2)$] $D(Q)$ equipped with with the inner product 
$$\as{f,g}_Q = Q(f,g) + \as{f,g} $$
is a Hilbert space. \hfill (Closedness)
\item[$(Q3)$] For any \emph{normal contraction} $F$ (i.e. a function $F:\RR \to \RR$ with $F(0)= 0$ and $|F(x)-F(y)| \leq |x-y|$ for all $x,y \in \RR$) and any $u \in D(Q)$ the function $F \circ u$ belongs to $D(Q)$ and the inequality
$$Q(F\circ u) \leq Q(u)  $$
holds. \hfill \text{(Markov~property)}
\end{itemize}
\end{definition}

We will call a Dirichlet form \emph{regular} if $C_c(V)$ is contained in $D(Q)$ and 
$$\overline{C_c(V)}^{\aV{\cdot}_Q} = D(Q),$$ 
where $\aV{\cdot}_Q$ is the norm given by 
$$\aV{\cdot}_Q = \sqrt{\as{\cdot,\cdot}}_Q.$$
\begin{remark}
 This definition of regularity differs from the one in \cite{FOT} which requires $D(Q) \cap C_c(V)$ to be dense in $D(Q)$ with respect to $\|\cdot\|_Q$ and in $C_c(V)$ with respect to $\|\cdot\|_\infty$. However, it was shown in \cite[Lemma~2.1]{KL} that a Dirichlet form on a countable discrete space is regular in the sense of \cite{FOT} if and only if the above is satisfied. 
\end{remark}

\begin{definition} 
A family $(T_t)_{t > 0}$ of bounded linear operators on the Hilbert space $\ell^2(V,m)$ is called \emph{strongly continuous Markovian semigroup} if the following conditions are satisfied: 
\begin{itemize}
\item[(S1)] For any $t>0$ the operator $T_t$ is self-adjoint.  (Symmetry)
\item[(S2)] $T_{t+s} = T_t T_s$ for every $t,s > 0$.  (Semigroup property)
\item[(S3)] $\aV{T_t f}_2\leq \aV{f}_2$ for every $t > 0, f\in \ell^2(V,m)$. (Contractivity)
\item[(S4)] $\aV{T_tf - f}_2 \to 0$ as $t \to 0$ for  every $f \in  \ell^2(V,m).$  (Strong continuity)
\item[(S5)] $0 \leq T_t f \leq 1$ for $f \in \ell^2(V,m)$ with $0\leq f \leq 1$. (Markov property)
\end{itemize}

A family $(G_\alpha)_{\alpha > 0}$ of bounded linear operators on $\ell^2(V,m)$ is called \emph{strongly continuous Markovian resolvent} if the following conditions are satisfied: 
\begin{itemize}
\item[(R1)] For any $\alpha > 0$ the operator $G_\alpha$ is self-adjoint. Symmetry)
\item[(R2)]$G_\alpha - G_\beta + (\alpha-\beta)G_\alpha G_\beta = 0$ for every $\alpha,\beta > 0$.  (Resolvent equation)
\item[(R3)] $\aV{\alpha G_\alpha f}_2\leq \aV{f}_2$ for every $\alpha > 0, f\in \ell^2(V,m)$.  (Contractivity)
\item[(R4)] $\aV{\alpha G_\alpha f - f}_2 \to 0$ as $\alpha \to \infty$ for  every $f \in  \ell^2(V,m)$. (Strong continuity)
\item[(R5)] $0 \leq \alpha G_\alpha f \leq 1$ for $f \in \ell^2(V,m)$ with $0\leq f \leq 1$. (Markov property)
\end{itemize}
\end{definition}

Every Dirichlet form $Q$ is in one to one correspondence with a non-negative self-adjoint operator $L$, a strongly continuous Markovian resolvent $G_\alpha$ and a strongly continuous Markovian semigroup $T_t$. Given any one of those four objects, one can reconstruct the others. We want to give a short discussion about the connection between those objects. For detailed proofs of the below statements see Chapter 1.3/1.4 of \cite{FOT}.
Given a Dirichlet form $Q$, the domain of its associated operator $L$ is given by
$$D(L) = \{ u\in D(Q)| \exists  w  \in \ell^2(V,m)\, \forall v\in D(Q):\, Q(u,v) = \as{w,v}\},$$
on which it  acts by $$Lu = w.$$ This operator is self-adjoint and nonnegative. Furthermore, its square root satisfies $D(L^{1/2}) = D(Q)$ and $$Q(u,v) = \as{L^{1/2}u, L^{1/2}v}$$ for all $u,v  \in D(Q)$. Because $L$ is nonnegative, its spectrum is contained in $[0,\infty)$. Thus, for positive $\alpha$ the operators $(L+\alpha)^{-1}$ exist and are bounded. They satisfy (R1)-(R5). The spectral calculus of $L$ allows us to define $e^{-tL}$ which is a semigroup satisfying (S1)-(S5). Let us stress some more relations of the above objects. 

\begin{itemize}
\item[(i)] Let $u \in \ell^{2}(V,m)$. Then,
$$Q(w,v) + \alpha \as{w,v} = \as{u,v}$$
holds for all $v \in D(Q)$ if and only if $w = (L+\alpha)^{-1}u$.

\item[(ii)] The domain of $L$ is given by 
$$D(L) = \left\{ u \in \ell^2(V,m) \mid \lim_{t\to 0}\frac{u-e^{-tL}u}{t} \text{ exists in } \ell^2(V,m)\right\}.$$
For $u \in D(L)$ the equality
$$Lu = \lim_{t\to 0}\frac{u-e^{-tL}u}{t}$$
holds, where the limit is taken in $\ell^2(V,m)$. 
\end{itemize}

Resolvents and semigroups associated with Dirichlet forms may be uniquely extended to bounded operators on $\ell^1(V,m)$ and $\ell^\infty (V)$. We discuss this extension for the resolvents, the semigroups can be treated similarly. Let $u\in \ell^1 (V,m)\cap \ell^2(V,m)$ and $K\subset V$ finite. Then, using the self-adjointness and the Markov property of $(L+\alpha)^{-1},$ we obtain
\begin{align*}
\sum_{x\in K} |(L+\alpha)^{-1}u(x)| m(x) &\leq \sumv{x} (L+\alpha)^{-1}|u|(x) 1_K(x)m(x) \\
& = \sumv{x} |u|(x) (L+\alpha)^{-1} 1_K(x)m(x)\\
&\leq \sumv{x} \alpha^{-1} |u|(x)m(x).
\end{align*}
Here $1_K$ denotes the indicator function of the set $K$. Since $K$ was  arbitrary, we infer $\aV{(L+\alpha)^{-1}u}_1 \leq \alpha^{-1} \aV{u}_1$. Thus, we can extend $(L+\alpha)^{-1}$ uniquely to $\ell^{1}(V,m)$. Now let $u \in \ell^\infty(V)$ be positive. Then, we can choose a sequence of non-negative functions $(u_n) \subseteq \ell^2(V,m)$ converging monotonously towards $u$. Because  $(L+\alpha)^{-1}$ maps non-negative functions onto non-negative functions we infer 
$$0 \leq (L+\alpha)^{-1}u_n \leq (L+\alpha)^{-1}u_{n+1}.$$
Furthermore, by (R5) we obtain 
$$(L+\alpha)^{-1}u_n \leq \alpha^{-1}\aV{u}_\infty.$$
Hence the limit as $n \to \infty$ exists and is bounded by $\alpha^{-1}\aV{u}_\infty$. It is easy to verify that this limit is independent of the choice of the sequence $u_n$. Now set
$$(L+\alpha)^{-1}u(x) :=  \lim_{n \to \infty} (L+\alpha)^{-1}u_n(x).$$
For an arbitrary $u \in \ell^\infty (V)$ split $u$ in its positive and negative part and repeat the above procedure. We then obtain a linear operator $(L+\alpha)^{-1}: \ell^\infty (V) \to \ell^\infty (V)$ satisfying
$$\aV{\alpha(L+\alpha)^{-1}u}_\infty \leq  \aV{u}_\infty.$$

\subsection{Dirichlet forms associated with graphs} \label{subsection:dirichlet forms of graphs}

\begin{theorem} \label{regular dirichlet forms}
Let $Q$ be a regular Dirichlet form on $\ell^2(V,m)$. Then, there exists a graph $(b,c)$ over $V$ such that $Q = \QD_{b,c}$. 
\end{theorem}
\begin{proof} Theorem 7 of \cite{KL}.
\end{proof}

The next Theorem is an approximation result for the resolvent $(\LD + \alpha)^{-1}$. We will need to fix some notation first. For finite $W \subseteq V$  let $L_{W} = p_W \ow{L}i_W :C(W) \to C(W)$. Here $i_W$ is the canonical embedding of $C(W)$ into $C(V)$ and $p_W$ the projection of $C(V)$ onto $C(W)$. In some sense $L_W$ is the restriction of $\ow{L}$ to $C(W)$, i.e. for $u \in C(W)$ and $x \in W$ we obtain
$$L_Wu(x) = \frac{u(x)}{m(x)}\sum_{y \in V}b(x,y) - \frac{1}{m(x)}\sum_{y \in W}b(x,y)u(y)+ \frac{c(x)}{m(x)} u(x). $$
Then, the following holds.

\begin{theorem} \label{approximation resolvent}
Let $\QD$ be the regular Dirichlet form associated with $(b,c)$ and let $\LD$ be the associated operator. Let $(K_n)$ be a sequence of finite subsets of $V$ such that  $K_n \subseteq K_{n+1}$ and $\cup K_n = V$. Then, for any $u \in C(K_1)$ 
$$\lim_{n \to \infty}\aV{(\LD + \alpha)^{-1}u -  (L_{K_n} + \alpha)^{-1}u}_2 = 0. $$
(Here $u$ and $(L_{K_n} + \alpha)^{-1}u$ are continued by $0$ outside of their domains).
\end{theorem}
\begin{proof} Proposition 2.7 of \cite{KL}.
\end{proof}

\begin{definition}

A bounded operator $T$ on $\ell^2(V,m)$ is called \emph{positivity improving} if $u \geq 0$ and $u \not \equiv 0$ implies $Tu (x) > 0$ for all $x \in V$.

\end{definition}

\begin{theorem}\label{positivity improving}
Let $(b,c)$ be connected and $Q$ be a Dirichlet form associated with $(b,c)$. Then, its corresponding resolvent $(L+\alpha)^{-1}$ and its corresponding semigroup $e^{-tL}$ are positivity improving.
\end{theorem}

\begin{proof}
Theorem 6.3 in \cite{HKLW}.
\end{proof}

\subsection{Vector valued integration}

The following theorems are special cases of statements for Bochner integrals.To avoid Hilbert space valued integration we include elementary proofs for them. Let us fix some notation. Let $ f:V \times [a,b] \to \RR$, such that for each $x \in V$ the function $f(x,\cdot)$ is integrable. We define
$$\int_a^b f(\cdot, t)dt : V \to \RR$$
pointwise via 
$$\int_a^b f(\cdot, t)dt(x):= \int_a^b f(x, t)dt.$$
Then, the following holds.

\begin{theorem} \label{integral inequality}
Assume $a,b \in \RR$. Let $f:V \times [a,b] \to \RR$, such that for each $t \in [a,b]$ the function $f(\cdot,t)$ belongs to $\ell^2(V,m)$  and $t \mapsto f(\cdot,t)$ is continuous as a mapping from $[a,b]$ to $\ell^2(V,m)$. Then, 
$$\aV{\int_a^b f(\cdot, t)dt}_2 \leq \int_a^b \aV{f(\cdot,t)}_2dt.$$
\end{theorem}

\begin{proof}
The continuity assumption ensures that all occurring integrals exist. Without loss of generality we can assume $m \equiv 1$. By monotone convergence it suffices to show the statement for finite sets $V$. This case can be reduced to $|V| =2$ by induction. Assume we have shown the result for all sets of cardinality less or equal to $n$, where $n \geq 2$, and let $|V| = n+1$. Fix $o \in V$. Then, 
\begin{align*}
 &\left(\sum_{x \in V}\left|\int_a^b f(x, t)dt\right|^2  \right)^{1/2}= \\
& = \left(\left|\int_a^b f(o, t)dt\right|^2 + \sum_{x \in V\setminus\{o\}}\left|\int_a^b f(x, t)dt\right|^2 \right)^{1/2}\\
 &\leq  \left(\left|\int_a^b f(o, t)dt\right|^2  + \left\{\int_a^b \left[\sum_{x \in V\setminus\{o\}} | f(x,t)|^2 \right]^{1/2} dt \right\}^2 \right)^{1/2}\\
 & \leq  \int_a^b \left(\sum_{x \in V} | f(x,t)|^2\right)^{1/2} dt = \int_a^b \aV{f(\cdot,t)}_2dt.
\end{align*}
Now let us treat the case $|V|= 2$ to finish the proof. Our continuity assumptions ensure that all of the above integrals can be computed via Riemann sums. Therefore, it suffices to show the statement for simple functions $f,g$ of the form
$$f = \sum_i \alpha_i 1_{A_i}$$
and 
$$g = \sum_i \beta_i 1_{A_i}$$
with pairwise disjoint sets $A_i$. We need to show
$$\left( \left| \int_a^b f(t)dt\right|^2 + \left| \int_a^b g(t)dt\right|^2\right)^{1/2} \leq \int_a^b \left(f(t)^2+g(t)^2\right)^{1/2}dt.$$
Plugging in $f$ and $g$ and taking the square on both sides of the above inequality, we conclude that it is equivalent to
$$\sum_{i,j}(\alpha_i\alpha_j + \beta_i\beta_j) \lambda (A_i)\lambda(A_j) \leq \sum_{i,j}(\alpha_i^2 + \beta_i^2)^{1/2}(\alpha_j^2 + \beta_j^2)^{1/2} \lambda (A_i)\lambda(A_j),$$
where $\lambda$ denotes the Lebesgue measure. But this inequality holds since 
$$\alpha_i\alpha_j + \beta_i\beta_j \leq (\alpha_i^2 + \beta_i^2)^{1/2}(\alpha_j^2 + \beta_j^2)^{1/2}$$
is always true (use that $(c-d)^2 \geq 0$ for arbitrary $c,d \in \RR$). This finishes the proof.
\end{proof}

\begin{theorem}\label{vector integration bounded operator}
Assume $a,b \in \RR$. Let $f:V \times [a,b] \to \RR$, such that for each $t \in [a,b]$ the function $f(\cdot,t)$ belongs to $\ell^2(V,m)$  and $t \mapsto f(\cdot,t)$ is continuous as a mapping from $[a,b]$ to $\ell^2(V,m)$. Furthermore, let $T$ be a bounded linear operator on $\ell^2(V,m)$. Then, $\int_a^b f(\cdot,t)dt \in \ell^2(V,m)$ and 
$$T \int_a^b f(\cdot,t)dt =  \int_a^b Tf(\cdot,t)dt. $$
\end{theorem}
\begin{proof}
$\int_a^b f(\cdot,t)dt \in \ell^2(V,m)$ follows from Theorem \ref{integral inequality} because $t \mapsto \aV{f(\cdot,t)}_2$ is continuous and
$$\aV{\int_a^b f(\cdot,t)dt}_2 \leq  \int_a^b \aV{f(\cdot,t)}_2dt.$$
Let $g \in \ell^2(V,m)$. Then, Lebesgue's theorem yields
$$\int_a^b \as{f(\cdot,t),g}dt = \as{\int_a^b f(\cdot,t)dt,g}.$$
Now the statement follows from
\begin{align*}
\as{T \int_a^b f(\cdot,t)dt, g} &= \as{\int_a^b f(\cdot,t)dt, T^*g}\\
&=\int_a^b \as{f(\cdot,t), T^*g}dt\\
&= \int_a^b \as{Tf(\cdot,t),g}dt\\
&= \as{\int_a^b Tf(\cdot,t)dt,g},
\end{align*}
where $T^*$ denotes the adjoint of $T$.
\end{proof}


%

\end{document}